\documentclass[11pt]{amsart}

\usepackage{amsmath,amssymb}
\usepackage[latin1]{inputenc}
\usepackage[dvips]{graphics}
\input{xy}
\xyoption{poly}
\xyoption{2cell}
\xyoption{all}
\usepackage{epsfig}  
\usepackage{pstricks}
\usepackage{color}

\def\Trop{\operatorname{Trop}}

\def\gg{\mathbf{g}}

\newcommand{\za}{\alpha}

\newcommand{\zD}{\Delta}

\newcommand{\zg}{\gamma}

\DeclareMathOperator{\UR}{ur}
\DeclareMathOperator{\tr}{tr}

\newcommand{\A}{\mathcal{A}}

\newcommand{\PP}{\mathbb{P}}

\newcommand{\R}{\mathbb{R}}

\newcommand{\Z}{\mathbb{Z}}

\newcommand{\Y}{\mathsf{y}}

\def\Aprin{\Acal_\bullet}

\def\Acal{\mathcal{A}}
\def\Fcal{\mathcal{F}}

\newtheorem{theorem}{Theorem}[section]

\newtheorem{lemma}[theorem]{Lemma}

\newtheorem{prop}[theorem]{Proposition}

\newtheorem{cor}[theorem]{Corollary}
\newtheorem{corollary}[theorem]{Corollary}

\theoremstyle{definition}
\newtheorem{definition}[theorem]{Definition}
\newtheorem{Def}[theorem]{Definition}
\newtheorem{example}[theorem]{Example}

\theoremstyle{remark}
\newtheorem{remark}[theorem]{Remark}

\numberwithin{equation}{section}

\begin{document}
\title{Matrix formulae and skein relations for cluster algebras from surfaces}
\author{Gregg Musiker}
\address{Department of Mathematics, University of Minnesota,
Minneapolis, MN 55455}
\email{musiker@math.umn.edu}
\thanks{{The first author is partially supported by NSF grant DMS-1067183.  The second author is partially supported by the 
NSF grant DMS-0854432 and a Sloan fellowship.}}
\author{Lauren Williams}
\address{Department of Mathematics, University of California,
Berkeley, CA 94720}
\email{williams@math.berkeley.edu}

\date{\today}
\dedicatory{}


\begin{abstract}
This paper concerns cluster algebras 
with principal coefficients
$\Aprin(S,M)$ associated to bordered surfaces  $(S,M)$,
and is a companion
to a concurrent work of the authors with Schiffler
\cite{MSW2}.  Given any 
(generalized) arc or loop in the surface
-- with or without self-intersections -- we associate
an element of (the fraction field of) $\Aprin(S,M)$,
using products of elements of $PSL_2(\R)$. 
We give a direct proof that our matrix formulas for 
arcs and loops
agree with the
combinatorial formulas for arcs and loops in terms of matchings, 
which were 
given in \cite{MSW, MSW2}.  
Finally, we use our matrix formulas to prove \emph{skein 
relations} for the cluster algebra elements associated
to arcs and loops.  Our matrix formulas and skein relations
generalize prior work of Fock and Goncharov 
\cite{FG1, FG2, FG3}, who worked in the coefficient-free case.
The results of this paper will be used in \cite{MSW2}
in order to show that certain collections of arcs and loops
comprise a vector-space basis for $\Aprin(S,M)$.

\end{abstract}
 
\maketitle
\setcounter{tocdepth}{1}
\tableofcontents

\section{Introduction}\label{intro}
Since their introduction by Fomin and Zelevinsky \cite{FZ1}, 
cluster algebras have been 
related to diverse areas of mathematics
such as total positivity, quiver representations, 
tropical geometry, Lie theory, Poisson geometry, and Teichm\"uller theory. 
There is an important
 class of cluster algebras arising from {\it bordered surfaces with marked
points}, introduced by Fomin, Shapiro, and Thurston in \cite{FST} 
(which in turn generalized work of Fock and Goncharov \cite{FG1, FG2} and Gekhtman, Shapiro, and 
Vainshtein \cite{GSV}), 
and further developed  in \cite{FT}.  
Such cluster algebras are interesting for several reasons: 
they comprise 
``most" of the mutation-finite skew-symmetric cluster 
algebras \cite{FeSTu}, and also 
they can be thought  of as coordinate rings for the 
\emph{decorated Teichm\"uller space} of $(S,M)$.
More specifically, 
the cluster variable associated to an arc in $(S,M)$ corresponds
to the {\it Penner coordinate} \cite{Pen} or
{\it exponentiated lambda length} of that arc.  

Because of the interpretation in terms of 
decorated Teichm\"uller space, one can express the Laurent
expansion of a cluster
variable associated to an arc in terms of a product of 
matrices in $PSL_2(\R)$; this was explained by Fock and
Goncharov in \cite{FG1, FG3}
in the coefficient-free case.
On the other hand, in our previous work with Schiffler \cite{MSW},
we gave formulas for the Laurent expansion of every cluster variable
in terms of perfect matchings of certain planar graphs.
We worked in the generality of principal coefficients,
and as a consequence proved the positivity conjecture for
cluster algebras from surfaces whose coefficient system is of geometric type.

Besides the positivity conjecture, a main open problem about cluster
algebras is to construct (vector-space) bases which have ``good" positivity
properties.  It is expected (and in many cases proved) that the 
cluster monomials are linearly independent \cite{CK,FZ4, DWZ}, 
and should be a part of 
such a basis, but in general, one needs some extra elements to complete
the cluster monomials to a basis.
In concurrent work \cite{MSW2}, we construct 
bases for the cluster algebras $\Aprin(S,M)$ from surfaces, with 
principal coefficients with respect to a seed $T$. 
In order to construct these bases and prove that they span the 
cluster algebra, 
we need to associate cluster
algebra elements $X_{\gamma}$ not only to arcs but also to generalized arcs 
and closed loops $\gamma$ 
(with self-intersections allowed).  Our formulas in \cite{MSW2} for such 
elements also involve matchings of certain graphs, but in the case 
of closed loops, these graphs lie on a M\"{o}bius strip or an annulus. 

In this paper we associate cluster algebra elements $\chi_{\gamma}$
to generalized
arcs and closed loops $\gamma$ using products of matrices in 
$PSL_2(\R)$. 
We prove that this definition using matrices
agrees with the definition using matchings, that is, that 
$\chi_{\gamma} = X_{\gamma}$ for any generalized arc or 
closed loop $\gamma$. 
In order to prove this, we 
prove a general combinatorial result
which explains how to enumerate matchings of a \emph{snake} or 
\emph{band} graph using products of $2 \times 2$ matrices.
We then prove \emph{skein relations},
which allow us to multiply elements
$\chi_{\alpha}$ and $\chi_{\beta}$, where $\alpha$ and $\beta$
are generalized arcs or closed loops.  Topologically, these
skein relations resolve crossings in the corresponding arcs or loops.

Note that 
our matrix formulas and skein relations
generalize prior work of Fock and Goncharov 
\cite{FG1, FG2, FG3}, who worked in the coefficient-free case.
However, it is crucial for us to work in the context of principal
coefficients, because our proofs in \cite{MSW2} use the notion 
of \emph{$\gg$-vectors}, which 
are defined in the case of principal 
coefficients.
Matrix formulas have also appeared in related literature, including \cite{String, ARS, BW, BW2, LambdaLengths, TropLambda}.

This paper is organized as follows. 
We give background on cluster 
algebras from surfaces
in Section \ref{sect surfaces}.  
In Sections \ref{sect main} and \ref{sec matrix} we give 
our combinatorial formulas from \cite{MSW2} and
our matrix formulas
for the cluster algebra 
elements associated to generalized arcs and closed loops.
In Section \ref{Sec:matrix=match} we prove that
these two formulas coincide.  Finally, in 
Section \ref{sec:skein} we prove the skein relations.

\vspace{1em}

\textsc{Acknowledgements:} 
We would like to thank Sergey Fomin, Ren Guo, Christophe Reutenauer, 
and Helen Wong for useful discussions.  
We are particularly grateful to Ralf Schiffler for his joint work 
with us, to Alexander Goncharov for inquiring about the connection
between his work and ours, and to Dylan Thurston for his inspirational lectures
in Morelia, Mexico.


\section{Cluster algebras arising from 
    surfaces}\label{sect surfaces} 

We assume that the reader is familiar with the notion of 
a cluster algebra and the terminology of \cite{FZ4}, including
principal coefficients and 
$F$-polynomials.  
We will begin by providing background on cluster algebras from surfaces.

Building on work of Fock and Goncharov \cite{FG1, FG2}, and of 
Gekhtman, Shapiro and Vainshtein \cite{GSV}, 
Fomin, Shapiro and Thurston \cite{FST} associated a cluster algebra
to any {\it bordered surface with marked points}.  In 
this section we will recall that construction, as well
as further results of Fomin and Thurston \cite{FT}.

\begin{Def}
[\emph{Bordered surface with marked points}]
Let $S$ be a connected oriented 2-dimensional Riemann surface with
(possibly empty)
boundary.  Fix a nonempty set $M$ of {\it marked points} in the closure of
$S$ with at least one marked point on each boundary component. The
pair $(S,M)$ is called a \emph{bordered surface with marked points}. Marked
points in the interior of $S$ are called \emph{punctures}.  
\end{Def}
 
For technical reasons, we require that $(S,M)$ is not
a sphere with one, two or three punctures;
a monogon with zero or one puncture; 
or a bigon or triangle without punctures.
 
\subsection{Ideal triangulations and tagged triangulations}

\begin{definition}
[\emph{Ordinary arcs}]
An \emph{arc} $\zg$ in $(S,M)$ is a curve in $S$, considered up
to isotopy, such that: 
the endpoints of $\zg$ are in $M$;
$\zg$ does not cross itself, except that its endpoints may coincide;
except for the endpoints, $\zg$ is disjoint from $M$ and
  from the boundary of $S$; and
$\zg$ does not cut out an unpunctured monogon or an unpunctured bigon. 
\end{definition}     

Curves that connect two
marked points and lie entirely on the boundary of $S$ without passing
through a third marked point are \emph{boundary segments}.
Note that boundary segments are not ordinary arcs.

\begin{Def}
[\emph{Crossing numbers and compatibility of ordinary arcs}]
For any two arcs $\zg,\zg'$ in $S$, let $e(\zg,\zg')$ be the minimal
number of crossings of 
arcs $\za$ and $\za'$, where $\za$ 
and $\za'$ range over all arcs isotopic to 
$\zg$ and $\zg'$, respectively.
We say that arcs $\zg$ and $\zg'$ are  \emph{compatible} if $e(\zg,\zg')=0$. 
\end{Def}

\begin{Def}
[\emph{Ideal triangulations}]
An \emph{ideal triangulation} is a maximal collection of
pairwise compatible arcs (together with all boundary segments). 
The arcs of a 
triangulation cut the surface into \emph{ideal triangles}. 
\end{Def}

There are two types of ideal triangles: triangles that have three distinct sides and triangles that have only two. The latter are called \emph{self-folded} triangles.  Note that a self-folded triangle consists of 
an arc $\ell$ whose endpoints coincide, together with an arc $r$ to an enclosed puncture which we dub a 
\emph{radius}.
Following the notation of \cite{GSV-book}, 
we will refer to an arc $\ell$ cutting out a once-punctured monogon as a \emph{noose}.

\begin{Def}
[\emph{Ordinary flips}]
Ideal triangulations are connected to each other by sequences of 
{\it flips}.  Each flip replaces a single arc $\gamma$ 
in a triangulation $T$ by a (unique) arc $\gamma' \neq \gamma$
that, together with the remaining arcs in $T$, forms a new ideal
triangulation.
\end{Def}

In a cluster algebra associated to an unpunctured surface, 
the cluster variables correspond to arcs, the clusters
to triangulations, and the mutations to flips.  However,
in a cluster algebra associated to a surface with punctures,
one needs to generalize the notion of arc and triangulation
in order to get a combinatorial framework that encodes the
whole cluster complex.
In \cite{FST}, the authors 
introduced 
{\it tagged arcs} and \emph{tagged triangulations}, and showed
that  they are in bijection with cluster variables and clusters.

\begin{Def}
[\emph{Tagged arcs}]
A {\it tagged arc} is obtained by taking an arc that is not a noose 
and marking (``tagging") each of its ends in one of two ways, {\it plain} or {\it notched},
so that the following conditions are satisfied:
\begin{itemize}
\item an endpoint lying on the boundary of $S$ must be tagged plain
\item if the endpoints of an arc coincide, then they 
must be tagged in the same way.
\end{itemize}
\end{Def}

\begin{Def}
[\emph{Representing ordinary arcs by tagged arcs}]
One can represent an ordinary arc $\beta$ by 
a tagged arc $\iota(\beta)$ as follows.  If $\beta$ is not a noose, 
then $\iota(\beta)$ is simply $\beta$ with both ends tagged plain.
Otherwise, $\beta$ is a noose based at point $a$,
which contains  the  puncture $b$ inside it.
Let $\alpha$ be the unique arc connecting $a$ and $b$ and compatible
with $\beta$.  Then $\iota(\beta)$ is obtained by tagging $\alpha$ plain at $a$ and notched at $b$.
\end{Def}

\begin{Def}
[\emph{Compatibility of tagged arcs}]  \label{compatible}
Tagged arcs $\alpha$ and
$\beta$ are called {\it compatible} if and only if the following 
properties hold:
\begin{itemize}
\item the arcs $\alpha^0$ and $\beta^0$ obtained from 
   $\alpha$ and $\beta$ by forgetting the taggings are compatible; 
\item if $\alpha^0=\beta^0$ then at least one end of $\alpha$
  must be tagged in the same way as the corresponding end of $\beta$;
\item $\alpha^0\neq \beta^0$ but they share an endpoint $a$, 
 then the ends of $\alpha$ and $\beta$ connecting to $a$ must be tagged in the 
same way.
\end{itemize}
\end{Def}

\begin{Def}
[\emph{Tagged triangulations}]  A maximal (by inclusion) collection
of pairwise compatible tagged arcs is called a {\it tagged triangulation}.
\end{Def}

\subsection{From  surfaces to cluster algebras}
One can  associate an exchange
matrix and
hence a cluster algebra to any bordered surface $(S,M)$
\cite{FST}.

\begin{Def}
[\emph{Signed adjacency matrix of an ideal triangulation}]
\label{adj-matrix}
Choose any ideal triangulation
$T$, and let $\tau_1,\tau_2,\ldots,\tau_n$ be the $n$ arcs of
$T$.
For any triangle $\Delta$ in $T$ which is not self-folded, we define a matrix 
$B^\Delta=(b^\Delta_{ij})_{1\le i\le n, 1\le j\le n}$  as follows.
\begin{itemize}
\item $b_{ij}^\Delta=1$ and $b_{ji}^{\Delta}=-1$ in the following cases:
\begin{itemize}
\item[(a)] $\tau_i$ and $\tau_j$ are sides of 
  $\Delta$ with  $\tau_j$ following $\tau_i$  in the 
  clockwise order;
\item[(b)] $\tau_j$ is a radius in a self-folded triangle enclosed by a noose $\tau_\ell$, and $\tau_i$ and $\tau_\ell$ are sides of 
  $\Delta$ with  $\tau_\ell$ following $\tau_i$  in the 
clockwise order;
\item[(c)] $\tau_i$ is a radius in a self-folded triangle enclosed by a noose $\tau_\ell$, and $\tau_\ell$ and $\tau_j$ are sides of 
  $\Delta$ with  $\tau_j$ following $\tau_\ell$  in the 
clockwise order;
\end{itemize}
\item $b_{ij}^\Delta=0$ otherwise.
\end{itemize}
 
Then define the matrix 
$ B_{T}=(b_{ij})_{1\le i\le n, 1\le j\le n}$  by
$b_{ij}=\sum_\Delta b_{ij}^\Delta$, where the sum is taken over all
triangles in $T$ that are not self-folded. 
\end{Def}

Note that $B_{T}$ is skew-symmetric and each entry  $b_{ij}$ is either
$0,\pm 1$, or $\pm 2$, since every arc $\tau$ is in at most two triangles.

\begin{theorem} \cite[Theorem 7.11]{FST} and \cite[Theorem 5.1]{FT}
\label{clust-surface}
Fix a bordered surface $(S,M)$ and let $\Acal$ be the cluster algebra associated to
the signed adjacency matrix of a tagged triangulation 
(see  \cite[Definition 9.18]{FST}).
Then the (unlabeled) seeds $\Sigma_{T}$ of $\Acal$ are in bijection
with tagged triangulations $T$ of $(S,M)$, and
the cluster variables are  in bijection
with the tagged arcs of $(S,M)$ (so we can denote each by
$x_{\gamma}$, where $\gamma$ is a tagged arc). Moreover, each seed in $\Acal$ is uniquely determined by its cluster.  Furthermore,
if a tagged triangulation $T'$ is obtained from another
tagged triangulation $T$ by flipping a tagged arc $\gamma\in T$
and obtaining $\gamma'$,
then $\Sigma_{T'}$ is obtained from $\Sigma_{T}$ by the seed mutation
replacing $x_{\gamma}$ by $x_{\gamma'}$.
\end{theorem}

\begin{remark}\label{ordinary-tagged}
By a slight abuse of notation, if $\gamma$ is an ordinary arc
which is not a noose 
(so that the tagged arc $\iota(\gamma)$ is obtained
from $\gamma$ by tagging both ends plain),  we will 
often write $x_{\gamma}$ instead of $x_{\iota(\gamma)}$.
\end{remark}

\begin{remark}
In this paper we will typically fix a triangulation 
$T=(\tau_1,\dots,\tau_n)$ of $(S,M)$.  The initial
cluster variables correspond to the arcs $\tau_i$,
and we will denote them by either
$x_{\tau_i}$ or $x_i$.  Similarly, we will denote
the initial coefficient variables by either
$y_{\tau_i}$ or $y_i$.
\end{remark}

Given a surface $(S,M)$ with a puncture $p$ and a tagged arc $\gamma$,
we let both $\gamma^{(p)}$ and $\gamma^{p}$
denote the arc obtained from $\gamma$
by changing its notching at $p$.  (So if $\gamma$ is not incident to
$p$, $\gamma^{(p)} = \gamma$.)
If $p$ and $q$ are two punctures, we let
$\gamma^{(pq)}$ denote the arc obtained from $\gamma$ by changing
its notching at both $p$ and $q$.
Given a tagged triangulation $T$
of $S$, we let $T^{p}$ denote the tagged triangulation obtained
from $T$ by replacing each $\gamma \in T$ by $\gamma^{(p)}$.

Besides labeling cluster variables of 
$\Acal(B_T)$ by $x_{\tau}$, where $\tau$ is a tagged arc of 
$(S,M)$, we will also make the following conventions:
\begin{itemize}
\item If $\ell$ 
is an unnotched noose with endpoints at $q$ cutting out a once-punctured monogon containing 
puncture $p$ and radius $r$, 
then we set 
$x_{\ell}= 
x_{r}x_{r^{(p)}}$. 
\item If $\beta$ is a boundary
segment, we set $x_{\beta} = 1$.
\end{itemize}

The exchange relation corresponding to a flip in an ideal triangulation
is called 
a {\it generalized Ptolemy relation}.  It can be described as 
follows.
\begin{prop}\cite{FT}\label{Ptolemy}
Let $\alpha, \beta, \gamma, \delta$ be arcs (including nooses) 
or boundary segments 
of $(S,M)$ which cut out a quadrilateral; we assume that the sides
of the quadrilateral, listed in cyclic order, are
$\alpha, \beta, \gamma, \delta$.  Let $\eta$ and $\theta$ 
be the two diagonals of this quadrilateral; see 
Figure \ref{figflip}.
Then 
\begin{equation} \label{shear-exchange} x_{\eta} x_{\theta} = Y x_{\alpha} x_{\gamma} + Y' x_{\beta} x_{\delta}\end{equation}
for some coefficients $Y$ and $Y'$.
\end{prop}

\begin{proof}
This follows from the interpretation of cluster variables as 
{\it lambda lengths}  and the 
Ptolemy relations for lambda lengths \cite[Theorem 7.5 and Proposition 6.5]{FT}.
\end{proof}

Note that 
some sides of the quadrilateral in Proposition \ref{Ptolemy}
may be glued to each other, changing the appearance of the relation.
There are also generalized Ptolemy relations for tagged triangulations,
see \cite[Definition 7.4]{FT}.
\begin{figure} \begin{center}
\scalebox{.8}{\input{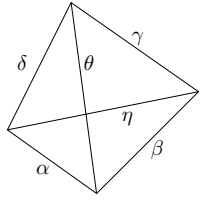}}
\end{center}
\caption{Quadrilateral illustrating Proposition \ref{Ptolemy}}
\label{figflip}
\end{figure}

\subsection{Principal coefficients}

In this paper we work with cluster algebras 
$\A = \Aprin(B_T)$ with \emph{principal coefficients}
with respect to the seed $\Sigma_T$.  Concretely,
these are defined by using a $2n \times n$ extended 
exchange matrix whose top $n \times n$ part is 
$B_T$, and whose bottom $n \times n$ part is the
identity matrix.  See \cite{FZ4} for more details.
When $B_T$ comes from a triangulation of a bordered
surface $(S,M)$, one can compute the coefficients using 
Thurston's theory of measured laminations; see 
\cite{FT} and also \cite{FG3}.  Concretely, one can 
compute principal coefficients with respect to 
the seed $\Sigma_T$ (where $T$ is a tagged triangulation)
using the \emph{shear coordinates} with respect to 
the $n$ \emph{elementary laminations} associated to the 
$n$ tagged arcs of $T$ \cite[Definition 16.2]{FT}.

For a cluster algebra $\mathcal{A}$ with exchange matrix $B_T$ and an arbitrary semifield $\PP$ 
of coefficients, Laurent expansions of cluster variables can be computed from the formula in $\Aprin(B,T)$ 
by the following theorem.

\begin{theorem}\cite[Theorem 3.7]{FZ4}
\label{th:reduction-principal}
Let $\Acal$ be a cluster algebra over an arbitrary semifield $\PP$
and contained in the ambient field $\Fcal$,
with a seed at an initial vertex $t_0$ given by
$$((x_1, \dots, x_n), (y_1^*, \dots, y_n^*), B^0).$$
Then the cluster variables in~$\Acal$ can be expressed as follows:
\begin{equation}
\label{eq:xjt-reduction-principal}
x_{\ell;t} = \frac{X_{\ell;t}^{B^0;t_0}|_\Fcal (x_1, \dots, x_n;y_1^*, \dots, y_n^*)}
{F_{\ell;t}^{B^0;t_0}|_\PP (y_1^*, \dots, y_n^*)} \, .
\end{equation}
\end{theorem}

An important class of semifields $\PP$ are the \emph{tropical semifields}.

\begin{Def}
[\emph{Tropical semifield}]
Let $\Trop (u_1, \dots, u_{m})$ be an abelian group (written
multiplicatively) freely generated by the $u_j$.
We define  $\oplus$ in $\Trop (u_1,\dots, u_{m})$ by
\begin{equation}
\prod_j u_j^{a_j} \oplus \prod_j u_j^{b_j} =
\prod_j u_j^{\min (a_j, b_j)} \,,
\end{equation}
and call $(\Trop (u_1,\dots,u_{m}),\oplus,\cdot)$ a \emph{tropical
 semifield}.
\end{Def}

A cluster algebra is of \emph{geometric type} whenever $\PP$ is such a semifield.
Notice that in this case,  the denominator of equation
(\ref{eq:xjt-reduction-principal}) is a monomial.

\section{Matching formulas for generalized arcs and closed loops} 
\label{sect main}

In this section we fix a bordered surface $(S,M)$, an ideal triangulation 
$T=(\tau_1,\dots,\tau_n)$, 
and the cluster algebra $\A = \Aprin(B_T)$ with principal coefficients
with respect to the seed $\Sigma_T$.  We will explain how to associate 
an element $X_{\gamma}^T$ of (the fraction field of) 
$\A$ to each generalized arc or closed loop $\gamma$ in $(S,M)$.
Our definition of the cluster algebra element associated to a  closed loop comes from joint work of the authors together with  Ralf Schiffler \cite{MSW2}.

Each element will be defined as a sum over matchings of a graph.
When $\gamma$ is an ordinary arc, $X_{\gamma}^T$ recovers the cluster 
expansion formula for the cluster variable associated to $\gamma$, with respect to $\Sigma_T$ \cite{MSW}.

\begin{definition} [\emph{Generalized arcs}] 
\label{gen-arc}
A \emph{generalized arc}  in $(S,M)$ is a curve $\gamma$ in $S$ such that:
the endpoints of $\gamma$ are in $M$; except for the endpoints,
$\gamma$ is disjoint from $M$ and the boundary of $S$; 
$\gamma$ does not cut out an unpunctured monogon or an unpunctured bigon.
Note that we allow a generalized arc to cross itself a finite 
number of times.  We consider generalized 
arcs up to isotopy (of immersed
arcs).  
In particular, an isotopy cannot remove a contractible
kink from a generalized arc.
\end{definition}

\begin{definition} [\emph{Closed loops}]
A closed loop in $(S,M)$ is a closed curve
$\gamma$ in $S$ which is disjoint from $M$ and the 
boundary of $S$.  We allow a closed loop to have a finite
number of self-intersections.
As in Definition \ref{gen-arc}, we consider closed
loops up to isotopy.
\end{definition}

\begin{definition}
A closed loop in $(S,M)$ is called \emph{essential} if:
\begin{itemize}
\item it is not contractible nor 
contractible onto a single puncture;
\item it does not have self-intersections.
\end{itemize}
\end{definition}

\subsection{Tiles} \label{sect tiles}
Let 
$\zg$ be a generalized  arc in $(S,M)$ which is not in $T$. 
Choose an orientation on $\zg$, let $s\in M$ be its starting point, and let $t\in M$ be its endpoint.
We denote by
$s=p_0, p_1, p_2, \ldots, p_{d+1}=t$
the points of intersection of $\zg$ and $T$ in order.  
Let $\tau_{i_j}$ be the arc of $T$ containing $p_j$, and let 
$\zD_{j-1}$ and 
$\zD_{j}$ be the two ideal triangles in $T$ 
on either side of 
$\tau_{i_j}$. 

To each $p_j$ we associate a  \emph{tile} $G_j$,  
an edge-labeled triangulated quadrilateral (see the right-hand-side of Figure \ref{digoncross}), 
which is defined to be
the union of two edge-labeled triangles $\zD_1^j$ and $\zD_2^j$ glued at an edge labeled $\tau_{i_j}$. 
The triangles $\zD_1^j$ and $\zD_2^j$ are determined by  
$\zD_{j-1}$ and 
$\zD_{j}$ as follows.

If neither $\zD_{j-1}$ nor $\zD_{j}$ is self-folded, then they each have three distinct sides 
(though possibly fewer than three vertices), and we define $\zD_1^j$ and $\zD_2^j$ to be the ordinary triangles with edges labeled as in $\zD_{j-1}$ and $\zD_{j}$.  We glue $\zD_1^j$ and $\zD_2^j$ at the edge labeled $\tau_{i_j}$, so that the orientations of $\zD_1^j$ and $\zD_2^j$ both either agree or disagree with those of $\zD_{j-1}$ and $\zD_j$; this gives two possible planar embeddings
of a graph $G_j$ which we call an \emph{ordinary tile}.
\begin{figure}
\input{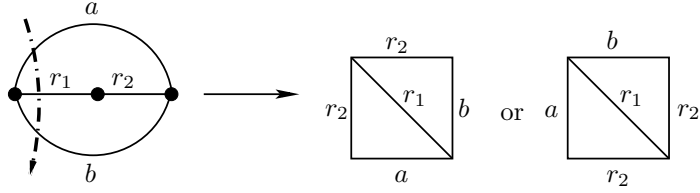}
\caption{Possible tiles corresponding to crossing radius of a bigon}
\label{digoncross}
\end{figure}

If one of $\zD_{j-1}$ or $\zD_{j}$ is  self-folded, then in fact $T$ must have a local configuration of a bigon (with sides $a$ and $b$) containing a radius $r$ incident to a puncture $p$ inscribed inside a loop $\ell$, see Figure \ref{triptile}.
If $\zg$ has no self-intersection as it passes
through the self-folded triangle, then $\zg$ must either

\begin{itemize}
\item[(1)] intersect the loop $\ell$ and terminate at puncture $p$, or

\item[(2)] intersect the loop $\ell$, radius $r$ and then $\ell$ again.
\end{itemize}

In case (1), we associate to $p_j$ (the intersection point with  $\ell$) 
an {\it ordinary tile} $G_j$ consisting of a triangle with sides $\{a,b,\ell\}$ which 
is glued along diagonal $\ell$ to a triangle with sides $\{\ell,r,r\}$.
As before there are two possible planar embeddings of $G_j$.

In case (2), we associate to the triple of intersection points 
$p_{j-1}, p_j, p_{j+1}$ a union of tiles $G_{j-1} \cup G_j \cup G_{j+1}$,
which we call a \emph{triple tile},
based on whether $\zg$ enters and exits
through different sides of the bigon or through the same side.
These graphs are defined by the first three examples in 
Figure \ref{triptile} (each possibility
is denoted in boldface within a concatenation of five tiles).  Note that 
in each case there are two possible planar embeddings of the triple tile.
We call
the tiles $G_{j-1}$ and $G_{j+1}$ within the triple tile {\it ordinary tiles}.

On the other hand, if $\zg$ has a self-intersection inside the self-folded triangle,
then the local configuration in the associated graph is as in the last
two examples of Figure \ref{triptile}.

\begin{figure}
\input{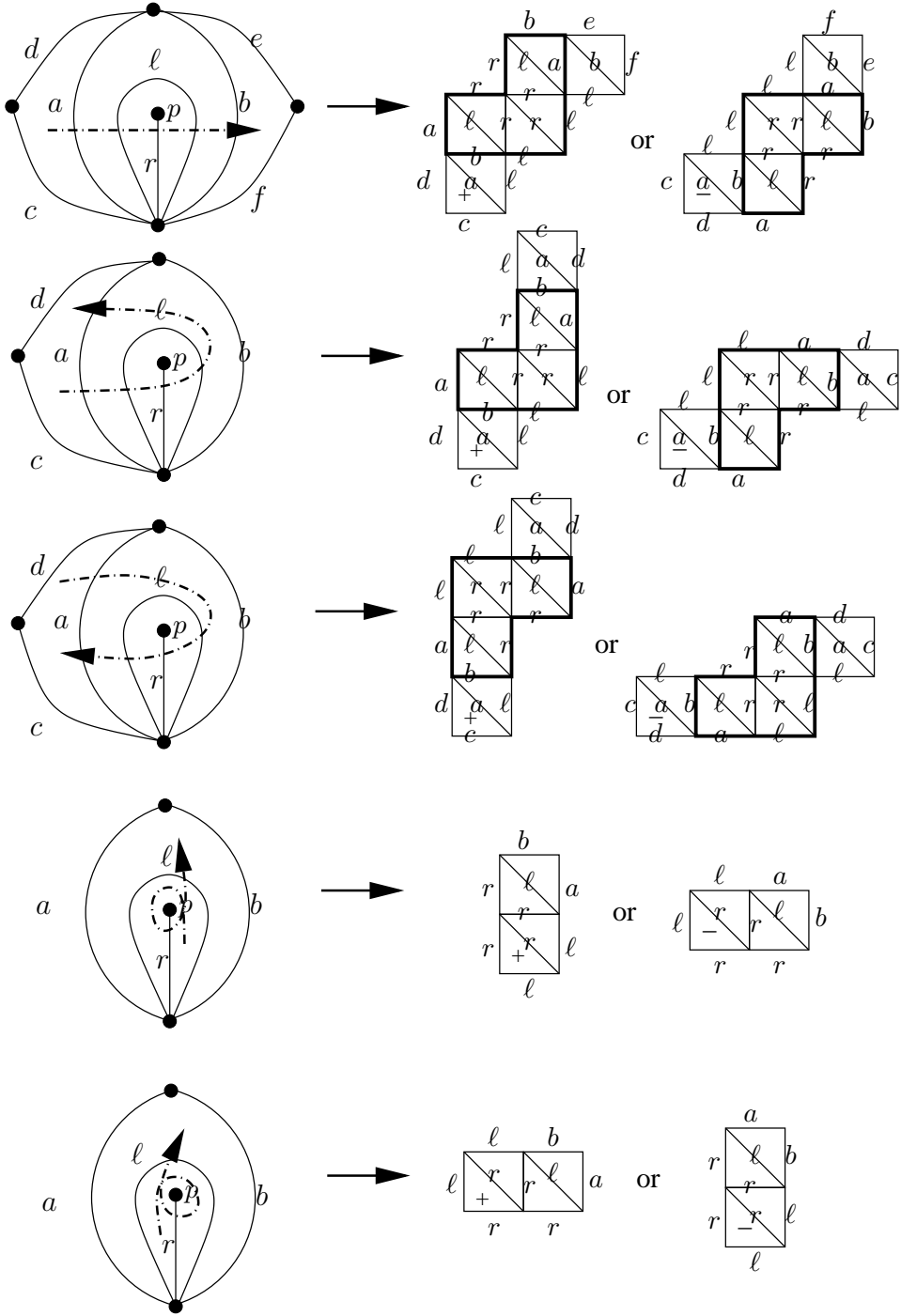}
\caption{Possible triple tiles for crossing a self-folded triangle}
\label{triptile}
\end{figure}

\begin{definition}[\emph{Relative orientation}] \label{def:relative-orientation}  
Given a planar embedding $\tilde G_j$ 
of an ordinary tile $G_j$, we define the \emph{relative orientation} 
$\mathrm{rel}(\tilde G_j, T)$ 
of $\tilde G_j$ with respect to $T$ 
to be $\pm 1$, based on whether its triangles agree or disagree in orientation with those of $T$.  
\end{definition}
Note that in Figure \ref{triptile}, the lowest tile in each of the five graphs in the middle
(respectively, rightmost) column has relative orientation $+1$ (respectively, $-1$),
as indicated by the signs in the figures.
Also note that 
by construction, the planar embedding of a triple tile $\tilde G_{j-1} \cup \tilde G_j \cup \tilde G_{j+1}$ satisfies $\mathrm{rel}(\tilde G_{j-1},T) = \mathrm{rel}(\tilde G_{j+1},T)$.

\begin{definition}
Using the notation above, 
the arcs $\tau_{i_j}$ and $\tau_{i_{j+1}}$ form two edges of a triangle $\zD_j$ in $T$.  Define $a_j$ to be the third arc in this triangle if $\zD_j$ is not self-folded, and to be the radius in $\zD_j$ otherwise.
\end{definition}

\subsection{The snake graph ${G}_{T,\zg}$}\label{sect graph}

We now recursively glue together the tiles $G_1,\dots,G_d$
in order from $1$ to $d$, subject to the following conditions.
\begin{figure}
\scalebox{0.7}{
\input{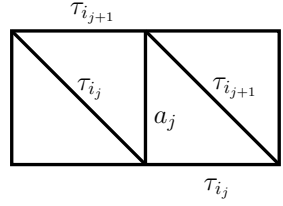}}
\caption{Glueing tiles $\tilde G_j$ and $\tilde G_{j+1}$ along the edge labeled  $a_j$}
\label{figglue}
\end{figure}
\begin{enumerate}
\item Triple tiles must stay glued together as  in Figure \ref{triptile}.
\item For two adjacent ordinary tiles, 
each of which may be an exterior tile of a triple tile, 
we glue  $G_{j+1}$ to $\tilde G_j$ along the edges 
labeled $a_j$, choosing a planar embedding $\tilde G_{j+1}$ for $G_{j+1}$
so that $\mathrm{rel}(\tilde G_{j+1},T) \not= \mathrm{rel}(\tilde G_j,T).$  See Figure \ref{figglue}.
\end{enumerate}

After glueing together the $d$ tiles, we obtain a graph (embedded in 
the plane),
which we denote 
$\overline{G}_{T,\zg}$.  Let $G_{T,\zg}$ be the graph obtained 
from $\overline{G}_{T,\zg}$ by removing the diagonal in each tile. We call $G_{T,\zg}$ the \emph{snake graph} associated to $\gamma$ with respect to $T$.
Figure \ref{triptile} gives examples of a 
dotted arc $\gamma$ and the corresponding graph 
$\overline{G}_{T,\zg}$.  Each $\gamma$ intersects $T$
five times,  so each 
$\overline{G}_{T,\zg}$ has five tiles.

\subsection{Definition of cluster algebra elements associated to generalized  arcs}\label{sect cluster expansion formula}

Recall that if $\tau$ is a boundary segment then $x_{\tau} = 1$,
and if $\tau$ is a noose  cutting out a once-punctured
monogon with radius $r$ and puncture $p$, then $x_{\tau} = x_r x_{r^{(p)}}$.

\begin{definition} [\emph{Crossing Monomial}]
If $\zg$ is an ordinary arc and  $\tau_{i_1}, \tau_{i_2},\dots, \tau_{i_d}$
is the sequence of arcs in $T$ which $\zg$ crosses, we define the \emph{crossing monomial} of $\gamma$ with respect to $T$ to be
$$\mathrm{cross}(T, \gamma) = \prod_{j=1}^d x_{\tau_{i_j}}.$$
\end{definition}

\begin{definition} [\emph{Perfect matchings and weights}]\label{def:perfect}
A \emph{perfect matching} of a graph $G$ is a subset $P$ of the 
edges of $G$ such that
each vertex of $G$ is incident to exactly one edge of $P$. 
If the edges of  a perfect matching $P$ of  
$G_{T,\zg}$ are labeled $\tau_{j_1},\dots,\tau_{j_r}$, then 
we define the {\it weight} $x(P)$ of $P$ to be 
$x_{\tau_{j_1}} \dots x_{\tau_{j_r}}$.
\end{definition}

\begin{definition} [\emph{Minimal and Maximal Matchings}]  
By induction on the number of tiles it is easy to see that 
$G_{T,\zg}$  
has  precisely two perfect matchings which we call
the {\it minimal matching} $P_-=P_-(G_{T,\zg})$ and 
the {\it maximal matching} $P_+
=P_+(G_{T,\zg})$, 
which contain only boundary edges.
To distinguish them, 
if  $\mathrm{rel}(\tilde G_1,T)=1$ (respectively, $-1$),
we define 
$e_1$ and $e_2$ to be the two edges of 
$\overline{G}_{T,\zg}$ which lie in the counter-clockwise 
(respectively, clockwise) direction from 
the diagonal of $\tilde G_1$.  Then  $P_-$ is defined as
the unique matching which contains only boundary 
edges and does not contain edges $e_1$ or $e_2$.  $P_+$
is the other matching with only boundary edges.
\end{definition}

For an arbitrary perfect matching $P$ of $G_{T,\gamma}$, we let
$P_-\ominus P$ denote the symmetric difference, defined as $P_-\ominus P =(P_-\cup P)\setminus (P_-\cap P)$.

\begin{lemma}\cite[Theorem 5.1]{MS}
\label{thm y}
The set $P_-\ominus P$ is the set of boundary edges of a 
(possibly disconnected) subgraph $G_P$ of $G_{T,\zg}$, which is a union of 
cycles.  These cycles enclose a set of tiles 
$\cup_{j\in J} G_{i_j}$,  where $J$ is a finite index set.
\end{lemma}

We use this decomposition to define \emph{height monomials} for
perfect matchings.
Note that the exponents in the height monomials defined below coincide
with the definition of height functions given in
\cite{ProppLattice} for perfect matchings of bipartite
graphs, based on earlier work of \cite{ConwayLagarias}, 
\cite{EKLP}, and \cite{Thurston} for domino tilings. 

\begin{definition} [\emph{Height Monomial and Specialized Height Monomial}] \label{height} Let $T = \{\tau_1,\tau_2,\dots, \tau_n\}$ be an ideal triangulation of $(S,M)$ and $\gamma$ be an ordinary arc of $(S,M)$.  By Lemma \ref{thm y}, for any perfect matching $P$
of $G_{T,\zg}$, $P \ominus P_-$ encloses the union of tiles $\cup_{j\in J} G_{i_j}$.  
We  define the \emph{height monomial} $h(P)$ of $P$ by
\begin{equation*}
h(P) = \prod_{k=1}^n \Y_{\tau_{k}}^{m_k},
\end{equation*}
where $m_k$ is the number of tiles in $\cup_{j\in J} G_{i_j}$ whose
diagonal is labeled $\tau_k$.

We  define the \emph{specialized height monomial} $y(P)$ of $P$ to be
the specialization $\Phi(h(P))$, where $\Phi$ is defined below. 
\begin{eqnarray*} 
\label{eqn yspec} \Phi({\Y}_{\tau_i}) &=& \left\{\begin{array}{ll}
y_{\tau_i}\
&\textup{if $\tau_i$ is not a side of a self-folded triangle;}\\ \\
\dfrac{y_{r}}{y_{r^{(p)}}}
&\textup{if $\tau_i$ is a radius $r$ to puncture $p$ in a self-folded triangle;}\\ \\
y_{r^{(p)}}
&\textup{if $\tau_i$ is a noose in a self-folded triangle with radius $r$ to puncture $p$.}
\end{array}\right.
\end{eqnarray*}
\end{definition}

\begin{definition}\label{e_p}
For an arc $\tau \in T$ and a puncture $p$,
let $e_p(\tau)$ denote the number of ends of $\tau$ incident to $p$ 
(so if both ends of $\tau$ 
are at $p$, $e_p(\tau) = 2$).
Additionally, if $\gamma$ is a generalized arc or loop, then 
$e(\tau,\gamma)$ denotes the number of crossings between $\tau$ and  $\gamma$.
\end{definition}

\begin{definition} \label{def:matching}
Let $(S,M)$ be a surface, $T=(\tau_1,\dots,\tau_n)$ an ideal triangulation,
and $\A = \Aprin(B_T)$ be the cluster algebra with principal coefficients
with respect to $\Sigma_T$.
Let $\zg$ 
be a generalized arc and let $G_{T,\zg}$ denote its
snake graph.  
We will define a Laurent polynomial $X_\zg^T$ 
which lies in (the fraction field of) $\A$, as well as a Laurent polynomial 
$F_\zg^T$ obtained from $X_{\gamma}^T$ by specialization.
\begin{enumerate}
\item If $\gamma$ 
cuts out a contractible monogon, then $X_\gamma^T$ 
is equal to zero.     
\item If $\zg$ 
has a contractible kink, let $\overline{\zg}$ denote the 
corresponding tagged arc with this kink removed, and define 
$X_{\zg}^T = (-1) X_{\overline{\zg}}^T$.  
\item Otherwise, define
\[ X_{\gamma}^T = \frac{1}{\mathrm{cross}(T,\zg)} \sum_P 
x(P) y(P),\]
 where the sum is over all perfect matchings $P$ of $G_{T,\zg}$.
\end{enumerate}
Define $F_{\gamma}^T$ to be the Laurent polynomial obtained from 
$X_{\gamma}^T$ by specializing all the $x_{\tau_i}$ to $1$.
\end{definition}

\begin{theorem}\cite[Thm 4.9]{MSW}
\label{thm MSW}
Use the notation of Definition \ref{def:matching}.
When $\gamma$ is an arc (with no self-intersections), 
$X_{\gamma}^T$ 
is the Laurent expansion of the cluster variable $x_{\gamma}\in \A$,
with respect to the seed $\Sigma_T$,
and $F_{\gamma}^T$ is its \emph{F-polynomial}.
\end{theorem}

\begin{remark}
In a few cases (see Lemma \ref{F-polynomial}), the elements
$F_{\gamma}^T$ may not be polynomials, only Laurent polynomials.
However, it's known that  a cluster algebra with principal coefficients
with respect to the seed $((x_1,\dots,x_n), (y_1,\dots,y_n),B)$ is 
contained in $\Z[x_1^{\pm},\dots,x_n^{\pm}; y_1,\dots,y_n]$ \cite[Proposition 3.6]{FZ4}.
Motivated by this fact and also Theorem \ref{th:reduction-principal}, 
we see that we should define the cluster algebra element associated to a generalized
arc  as in Definition \ref{rightdefinition}.
\end{remark}

\begin{definition}\label{rightdefinition}
Let $\PP$ be the tropical semifield $Trop(y_1,\dots, y_n)$.
We define the cluster algebra element $x_{\gamma}^T \in \A=\Aprin(B_t)$ associated to 
a generalized arc  by 
\begin{equation}
x_{\gamma}^T = \frac{X_{\gamma}^{T}}
{F_{\gamma}^T|_\PP (y_1, \dots, y_n)} \, .
\end{equation}
\end{definition}

\begin{remark}
Most of the time (see Lemma \ref{F-polynomial}) $F_{\gamma}^T$
is a polynomial with constant term $1$, and hence 
$x_{\gamma}^T = X_{\gamma}^T$.
\end{remark}

\subsection{Definition of the cluster algebra elements associated to closed loops}\label{sect def loops}

By using a variant of the above construction, for each closed loop $\gamma$ we will 
associate an element of (the fraction field of) $\A$ \cite{MSW2}.

\begin{definition} [\emph{Band Graph corresponding to a closed loop $\gamma$}]
\label{def band}
Let $\gamma$ be a closed loop in $(S,M)$, which may or not have self-intersections, 
but is not contractible and has no contractible kinks.  We pick a triangle $\Delta$ traversed by $\gamma$ arbitrarily, let $p$ be a point in the interior of this triangle which lies on $\gamma$, and let $b$ and $c$ be the two sides of triangle crossed by $\gamma$ 
immediately before and following its travel through point $p$, respectively.  
Let $a$ be the third side of $\Delta$.   Note that these definitions
 make sense even if $\Delta$ is self-folded.  
We let $\tilde{\gamma}$ denote the arc from $p$ back to itself that exactly 
follows closed loop $\gamma$.  See the left of Figure \ref{fig band}.

\begin{figure}
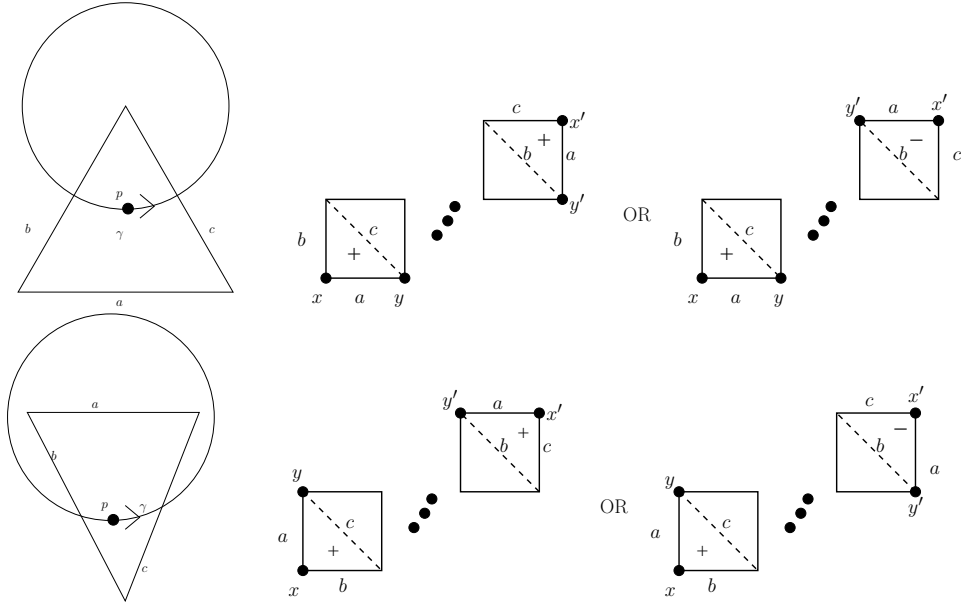

\scalebox{0.2}
{
\input{BandTriag.pstex_t}
\hspace{8em}
}
\scalebox{0.3}
{
\input{BandGs.pstex_t}
}
\scalebox{0.2}
{
\input{BandTriag2.pstex_t}
\hspace{8em}
}
\scalebox{0.3}
{
\input{BandGs2.pstex_t}
}
\caption{(Left): Triangle containing $p$ along closed loop $\gamma$. \hspace{10em}
(Right): Corresponding Band graph (with $x\sim x'$, $y\sim y'$) depending on whether $\gamma$ crosses an odd or even number of arcs. The $+$'s and $-$'s denote relative orientation of each tile}
\label{fig band}
\end{figure}

We start by building the snake graph $G_{T, \tilde{\gamma}}$
as defined above.  In the first tile of $G_{T, \tilde{\gamma}}$, let $x$ denote the vertex 
at the corner of the edge labeled $a$ and the edge labeled $b$, and let $y$ denote the vertex at the other end of the edge labeled $a$.  Similarly, in the last tile of 
$G_{T, \tilde{\gamma}}$, let $x'$ denote the vertex at the corner of the edge labeled 
$a$ and the edge labeled $c$, and let $y'$ denote the vertex at the other 
end of the edge labeled $a$.  See the right of Figure \ref{fig band}.

From $G_{T, \tilde{\gamma}}$, we build $\widetilde{G}_{T, \gamma}$, the band graph 
for the closed loop $\gamma$, by identifying the edges labeled $a$ in the first and last tiles so that the vertices $x$ and $x'$ and the vertices $y$ and $y'$ are glued together.  We refer to the two vertices obtained by identification
as $x$ and $y$, and to the edge obtained by identification as the \emph{cut edge}.
The resulting graph lies on an annulus or a M\"{o}bius strip.  
\end{definition}

\begin{definition} [\emph{Good matchings on a band graph}]
\label{def good matching}
Let $P$ be a perfect matching of a band graph $\widetilde{G}$.  We call $P$ \emph{good} 
if either $x$ and $y$ are matched to each other ($P(x)=y$ and $P(y)=x$) or if both 
edges $(x,P(x))$ and $(y,P(y))$ lie on one side of the cut edge.
\end{definition}

\begin{remark}\label{descend}
Let $\widetilde{G}$ be a band graph obtained by identifying two edges
of the snake graph $G$.  The good matchings of
$\widetilde{G}$ can be identified with a subset of the perfect matchings
of $G$.  Let $\widetilde{P}$ be a good matching of $\widetilde{G}$.
Thinking of $\widetilde{P}$ as a subset of edges of $G$, then
by definition of good we can add to it either the edge $(x,y)$ or
the edge $(x',y')$ to get a perfect matching $P$ of $G$.  In this case,
we say that the perfect matching $P$ of $G$ \emph{descends} to a
good matching $\widetilde{P}$ of $\widetilde{G}$.  In particular,
the minimal matching of $G$ descends to a good matching of
$\widetilde{G}$, which we also call \emph{minimal}.
(To see this, just consider the cases of whether $G$ has an odd or
even number of tiles, and observe that the minimal matching
of $G$ always uses one of the edges $(x,y)$ and $(x',y')$.)
\end{remark}

\begin{definition} \label{def closed loop}
Let $(S,M)$ be a surface, $T=(\tau_1,\dots,\tau_n)$ an ideal triangulation,
and $\A = \Aprin(B_T)$ be the cluster algebra with principal coefficients
with respect to $\Sigma_T$.
Let $\zg$ 
be a closed curve. 
We define a Laurent polynomial $X_\zg^T$
which lies in (the fraction field of) $\A$, as well as a Laurent polynomial 
$F_\zg^T$ obtained from $X_{\gamma}^T$ by specialization.
\begin{enumerate}
\item If $\zg$ is a contractible loop,
 then let $X_\zg^T = -2$.
\item If $\zg$ is a closed loop without self-intersections 
enclosing a single puncture $p$:
\begin{itemize}
\item If $T$ contains a self-folded triangle containing $p$,
then let $X_{\gamma}^T = 1 + \frac{y_r}{y_r^{(p)}}$, 
where $r$ is the 
radius incident to $p$.
\item Otherwise, let 
$X_\gamma^T = 1 + \prod_{\tau\in T} y_\tau^{e_p(\tau)}$, where $e_p(\tau)$ is given by Definition \ref{e_p}.\footnote{We could instead
define $X_{\gamma}^T$ in terms of matchings in a band graph,
as below, but
this definition in terms of $e_p$ is simpler to compute.}
\end{itemize}
\item If $\zg$ 
has a contractible kink, let $\overline{\zg}$ denote the 
corresponding closed curve with this kink removed, and define 
$X_{\zg}^T = (-1) X_{\overline{\zg}}^T$.
\item Otherwise, let 
$$X_{\gamma}^T = \frac{1}{\mathrm{cross}(T,\zg)} \sum_{P} 
x(P) y(P),$$ where the sum is over all good matchings $P$ of the  band graph $\widetilde{G}_{T,\zg}$.
\end{enumerate}
Define $F_{\gamma}^T$ to be the Laurent polynomial obtained from 
$X_{\gamma}^T$ by specializing all the $x_{\tau_i}$ to $1$.
\end{definition}

\begin{remark}\label{rightdefinition2}
We also apply Definition \ref{rightdefinition} to define the 
cluster algebra element 
$x_{\gamma}^T \in \A=\Aprin(B_t)$ when $\gamma$ is a closed loop.
\end{remark}

\begin{lemma}\label{F-polynomial}
If $T$ has no self-folded triangles, then for any generalized arc or 
loop $\gamma$, $F_{\gamma}^T$ is 
a polynomial with constant term $1$.
If $T$ has self-folded triangles, then for any arc (without self-intersections)
or essential loop $\gamma$,
$F_{\gamma}^T$ is a polynomial with constant term $1$.
\end{lemma}
\begin{proof}
By definition, $F_{\gamma}^T = \sum_P \Phi(h(P))$, where the sum 
is over all matchings (resp. good matchings) of the snake graph (resp.
band graph) of $\gamma$.
Clearly $\sum_P h(P)$ is a polynomial in the variables 
$\Y_{\tau_i}$, with constant term $1$.  
The only way that Proposition \ref{F-polynomial} might fail
is if the specialization $\Phi$ produces a Laurent monomial
which is not a monomial, or an extra constant term.  However,
whenever $\gamma$ is an arc or an essential loop, 
each time $\gamma$ goes through
a noose, its local configuration must look like one of the first
three diagrams in 
Figure \ref{triptile}.  By inspection, it's impossible for 
the height monomial $h(P)$ of a matching $P$ of 
the corresponding band graph to include a contribution from 
the tile labeled $r$ \emph{without} also including a contribution
from one of the two adjacent tiles labeled $\ell$.  Therefore 
$\Phi(h(P))$ will never produce a denominator.  Also, it's impossible
for $\Phi(h(P))$ to produce an addition term equal to $1$: such a term
would need to have a factor of $\frac{y_r}{y_{r^{(p)}}}$, but then
there would be no way to cancel the numerator $y_r$.
\end{proof}

We end this section with two examples illustrating the computation of $X_{\gamma}^T$.

\begin{example}
[Example of a Laurent expansion corresponding to a closed loop]
\label{ExampleBand}

\begin{figure}
\input{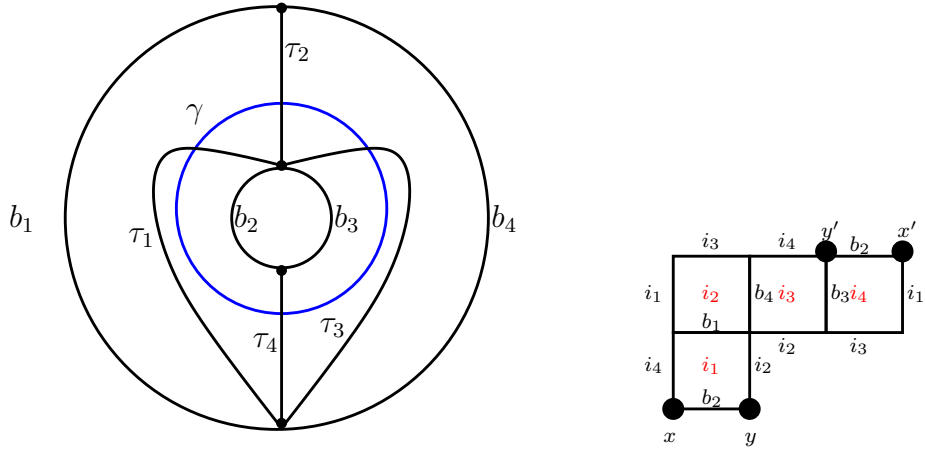}
\caption{(Left): An ideal triangulation $T$ and a closed loop $\gamma$. (Right) Corresponding band graph $\tilde{G}_{T,\gamma}$}
\label{fig Ex-band}
\end{figure}

Consider the ideal triangulation $T$ and closed loop $\gamma$ on the left of Figure \ref{fig Ex-band}.  In this case, we obtain the band graph $\tilde{G}_{T,\gamma}$ appearing on the right of Figure \ref{fig Ex-band}.  We thus compute 
$X_\gamma^T = \frac{x_1^2x_2x_4 + y_3x_1^2 + (y_2y_3 + y_3y_4)x_1x_3 + y_2y_3y_4x_3^2 + y_1y_2y_3y_4x_2x_3^2x_4}{x_1 x_2 x_3 x_4}$ by specializing $b_1=b_2=b_3=b_4=1$.  In particular, $\tilde{G}_{T,\gamma}$ has six good matchings, as listed in Figure \ref{fig Ex-band2}.   

\begin{figure}
\input{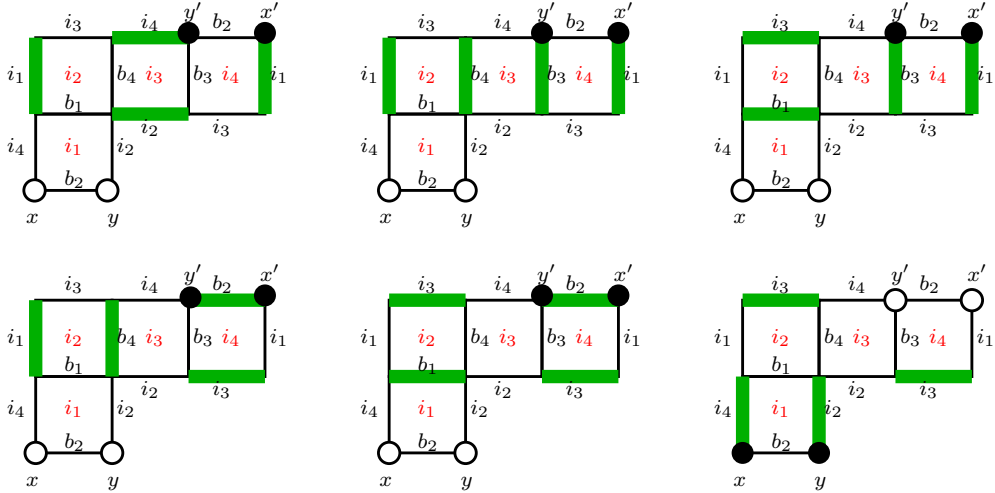}
\caption{The good matchings in the band graph $\tilde{G}_{T,\gamma}$ of the above example}
\label{fig Ex-band2}
\end{figure}

\end{example}

\begin{example}
[Examples of Laurent expansions for generalized arcs through a self-folded triangle]
\label{ExampleFolded}

\begin{figure}
\input{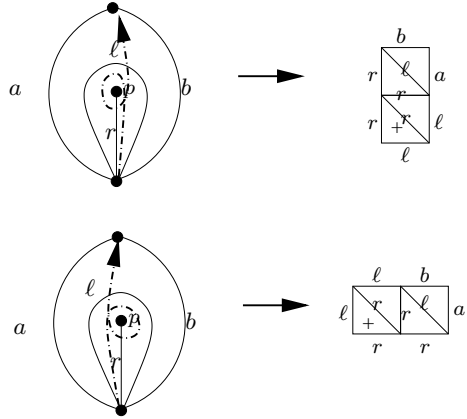}
\caption{(Left): An ideal triangulation $T$ containing a self-folded triangle and generalized arcs $\gamma_1$ and $\gamma_2$. (Right): Corresponding snake graphs $G_{T,\gamma_1}$ and $G_{T,\gamma_2}$}
\label{fig Ex-Folded}
\end{figure}

Consider the ideal triangulation $T$ and the generalized arc $\gamma_1$ on the top (resp. $\gamma_2$ on the bottom) on the left of Figure \ref{fig Ex-Folded}.  In this case, we obtain the snake graph $G_{T,\gamma_1}$ (resp. $G_{T,\gamma_2}$) appearing on the right of Figure \ref{fig Ex-Folded}.  We thus compute 
$$X_{\gamma_1}^T = \frac{a x_r x_\ell + \Y_\ell b x_r x_\ell + \Y_r \Y_\ell b x_r x_\ell}{x_r x_\ell} = b(y_{\tau^{(p)}}+ y_\tau) + a$$ and 
$$X_{\gamma_2}^T = \frac{a x_r x_\ell + \Y_r a x_r x_\ell + \Y_r \Y_\ell b x_r x_\ell}{x_r x_\ell} = a(1 + \frac{y_\tau}{y_{\tau^{(p)}}}) + y_\tau b$$
by specializing $\Y_\ell = y_{\tau^{(p)}}$, $\Y_r = \frac{y_\tau}{y_{\tau^{(p)}}}$.

Notice that in the case of $X_{\gamma_2}^T$, we obtain a Laurent polynomial expansion with a $y_{\tau^{(p)}}$ in the denominator.  The corresponding algebraic quantity $F_{\gamma_2}^T = 1 + \frac{y_\tau}{y_{\tau^{(p)}}} + y_\tau$ is in fact a \emph{Laurent} polynomial and not a polynomial in this case.  To obtain the associated cluster algebra element $x_{\gamma_2}^T$, we follow
Definition \ref{rightdefinition}, and 
divide $X_{\gamma_2}^T$ by the tropical evaluation of 
$F_{\gamma_2}^T$, which is $1/y_{\tau^{(p)}}$. This gives us 
$x_{\gamma_2}^T = a(y_{\tau^{(p)}} + y_\tau) + b y_\tau y_{\tau^{(p)}}$.
\end{example}


\section{Matrix product formulas for generalized arcs and closed loops}
\label{sec matrix}

\newcommand{\Matrix}{M}

\newcommand{\MatrixA}{\chi}
\newcommand{\MatrixL}{\chi}
\newcommand{\GG}{\mathcal{G}}

In this section we also fix a bordered surface $(S,M)$ and ideal triangulation $T$ of $(S,M)$.
We then associate a Laurent polynomial to each arc, generalized arc,
and closed loop in $(S,M)$; this Laurent polynomial represents the corresponding
element of the cluster algebra $\A_T(S,M)$ with principal coefficients with respect
to the seed $T$.
Our formula works by associating a product of matrices
to each such arc (respectively, loop), and then computing  its upper right entry 
(respectively, trace).  
Our formulas are closely related to those given in  
\cite{FG1} and \cite{FG3}.  In particular, if we set our initial cluster variables
equal to $1$, we recover Fock and Goncharov's {\it $X$-coordinates}, and if we set our
coefficient variables equal to $1$, we recover their {\it $A$-coordinates}.

Before presenting our formulas, we need to define some \emph{elementary steps},
and the matrices associated to them.

\begin{definition} (The points $v_{m,\tau}$, $v_{m,\tau}^+$, $v_{m,\tau}^-$)
For each marked point $m \in M$, draw a small 
circle $h_m$ 
locally around $m$.  If $m$ is on the boundary of $S$, then we only consider the intersection $h_m \cap S$.  The circles are chosen small enough so that $h_m \cap h_{m'} = \emptyset$ for each pair of distinct marked points $m$ and $m'$.
For each arc $\tau \in T$ and marked point $m \in M$ incident to $\tau$, we let $v_{m,\tau}$ denote the intersection point $h_m \cap \tau$.  
We let $v_{m,\tau}^+$ (resp. $v_{m,\tau}^-$) denote a point on $h_m$ which is very close
to $v_{m,\tau}$ but in the clockwise (resp. counterclockwise) direction from $v_m$. 
See Figure \ref{fig infinitesimal}.  
\end{definition}

\begin{figure}
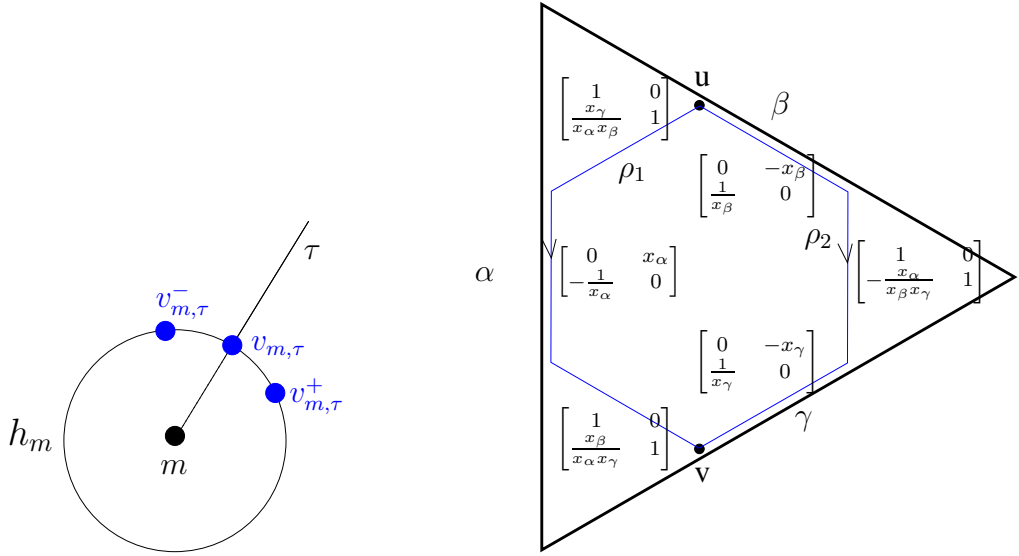

\input{Horocnew.pstex_t} \hspace{5em} \input{TriagHexNew.pstex_t}
\caption{(Left): $h_m$ and $v_{m,\tau}^{\pm}$. (Right): two $M$-paths from $u$ to $v$}
\label{fig infinitesimal}
\end{figure}

\begin{definition}(Elementary steps)
Given $(S,M)$ and $T$, we define three types of \emph{elementary steps}, each of which
connects two points of the form $v_{m,\tau}^{\pm}$ and $v_{m',\tau'}^{\pm}$.  We also
associate a $2 \times 2$ matrix $M(\rho)$ to each elementary step $\rho$.
\begin{itemize}
\item 
The first type of step is shown in Figure \ref{fig step12}.  We consider two 
arcs $\tau$ and $\tau'$ from $T$ which are both incident to a marked point $m$ and which form 
a triangle with  third side $\sigma\in T$.  Then our first type of step
is a curve 
which travels either clockwise or counterclockwise around $h_m$ between 
$\tau$ and $\tau'$, without crossing them. 
The matrix corresponding to this step is
$\left[ \begin{matrix} 1 & 0 \\ \pm \frac{x_{\sigma}}{x_{\tau}x_{\tau'}} & 1 \end{matrix}\right]$, where we choose the positive (resp. negative) sign if the step is clockwise (resp. counterclockwise).  
\item The second type of step is shown in  Figure \ref{fig step34}.  This step moves along a circle $h_m$ connecting 
two points $v_{m,\tau}^+$ and $v_{m,\tau}^-$, so in particular it crosses the arc $\tau$. If the step 
travels clockwise (resp. counterclockwise), we associate to it the matrix
$\left[ \begin{matrix} 1 & 0 \\ 0 & \Y_{\tau} \end{matrix}\right]$ (resp., 
$\left[ \begin{matrix} \Y_{\tau} & 0 \\ 0 & 1 \end{matrix}\right]$).  
\item
The third type of step is shown in Figure \ref{fig step56}.  Given two marked points $m$ and $m'$ connected by some 
$\tau \in T$, such a step follows a path parallel to $\tau$, and  connects  
$v_{m,\tau}^{\pm}$ and $v_{m',\tau}^{\mp}.$  
We associate the matrix $\left[ \begin{matrix} 0 & x_{\tau} \\ -\frac{1}{x_{\tau}} & 0 \end{matrix}\right]$ to such a step if 
$\tau$ lies beneath it (when we orient the step from left to right), and we associate the inverse matrix $\left[ \begin{matrix} 0 & -x_{\tau} \\ \frac{1}{x_{\tau}} & 0 \end{matrix}\right]$ to the step otherwise.
\end{itemize}
\end{definition}

\begin{figure}
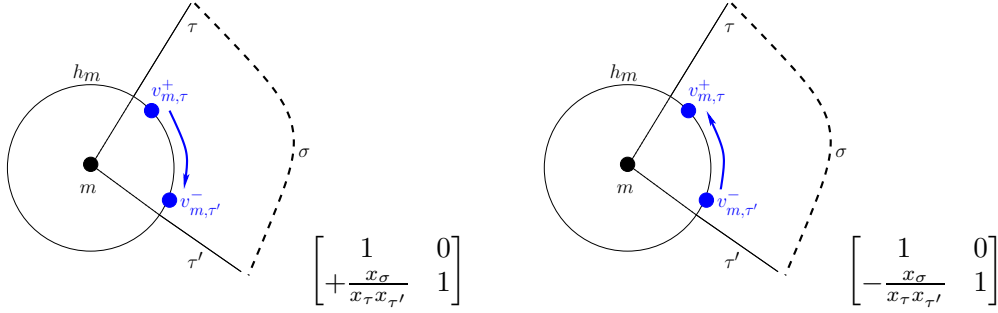

\scalebox{0.3}{\input{Step1new.pstex_t}} $\left[ \begin{matrix} 1 & 0 \\ + \frac{x_{\sigma}}{x_{\tau} x_{\tau'}} & 1 \end{matrix}\right]$ \hspace{2em} 
\scalebox{0.3}{\input{Step2new.pstex_t}} $\left[ \begin{matrix} 1 & 0 \\ - \frac{x_{\sigma}}{x_{\tau}x_{\tau'}} & 1 \end{matrix}\right]$
\caption{Elementary step of type 1}
\label{fig step12}
\end{figure}

\begin{figure}
\scalebox{0.3}{\input{Step3new.pstex_t}} $\left[ \begin{matrix} 1 & 0 \\ 0 & \Y_{\tau} \end{matrix}\right]$ \hspace{2em} 
\scalebox{0.3}{\input{Step4new.pstex_t}} $\left[ \begin{matrix} \Y_{\tau} & 0 \\ 0 & 1 \end{matrix}\right]$
\caption{Elementary step of type 2}
\label{fig step34}
\end{figure}

\begin{figure}
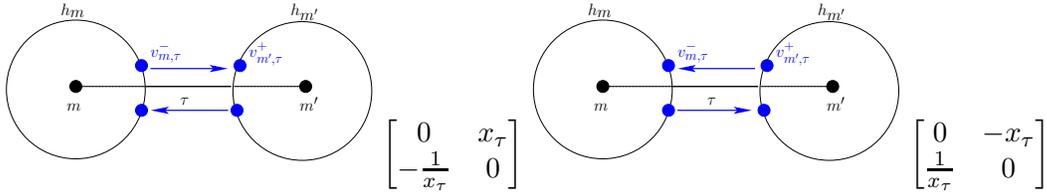

\scalebox{0.25}{\input{Step5Final.pstex_t}} $\left[ \begin{matrix} 0 & x_{\tau} \\ -\frac{1}{x_{\tau}} & 0 \end{matrix}\right]$
\scalebox{0.25}{\input{Step6Final.pstex_t}}
 $\left[ \begin{matrix} 0 & -x_{\tau} \\ \frac{1}{x_{\tau}} & 0 \end{matrix}\right]$
\caption{Elementary steps of type 3 in the positive and negative direction}
\label{fig step56}
\end{figure}

\begin{remark} \label{rem:not_in_SL2}
Note that the two matrices associated to elementary steps of type 1 and elementary steps of type 3 are inverses in $SL_2(\R)$.   Also, the product of the two
matrices associated to elementary steps of type 2 is $\Y_{\tau}$ times the identity.
\end{remark}

We are now ready to associate a matrix to each generalized arc and closed loop.

\begin{definition} (Matrix associated to an arc or loop)\label{matrixforarc}
Given $(S,M)$, $T$, and a generalized arc $\gamma$ in $(S,M)$ from $s$ to $t$, we 
choose a curve $\rho_{\gamma}$ which has the following properties:
\begin{itemize} 
\item It begins at a point $P_s$ of the form 
$v_{s,\tau}^{\pm}$, where $\tau$ is an arc of $T$ incident to $s$.  
\item It ends at a point $P_t$
of the form $v_{t,\tau'}^{\pm}$.  
\item It is a 
concatenation of elementary steps, and is isotopic to 
the portion of $\gamma$ between $h_s \cap \gamma$ and $h_t \cap \gamma$.
\item The intersections of $\rho_{\gamma}$ with $T$ 
are in bijection with the intersections of $\gamma$ with $T$. 
\end{itemize}
Given a closed loop $\gamma$ in $(S,M)$, we choose a curve $\rho_{\gamma}$
which has the following properties:
\begin{itemize}
\item It starts and ends at a point $P_s=P_t$ which has the form 
$v_{m,\tau}^{\pm}$, where $\tau \in T$ is crossed by $\gamma$.
\item It is a concatenation of elementary steps, and is isotopic to $\gamma$.
\item The intersections of $\rho_{\gamma}$ with $T$ are in bijection with 
the intersections of $\gamma$ with $T$.
\end{itemize}
In both cases, we refer to the curve $\rho_{\gamma}$ as an \emph{$M$-path}.
If $\rho_{\gamma} = \rho_t \circ \dots \circ \rho_2 \circ \rho_1$ is the decomposition of $\rho_{\gamma}$ into elementary 
steps, then we define 
$M(\rho_{\gamma})
= M(\rho_t) \cdots M(\rho_2) M(\rho_1).$  By convention, the matrix associated to the empty path is the identity matrix.
\end{definition}

\begin{definition} (Reduced Matrix associated to an arc or loop)
\label{def:reduced_matrix}
As noted in Remark \ref{rem:not_in_SL2}, the matrices corresponding to elementary steps of type $2$ are not in $SL_2(\R)$.  For some applications, 
it will be more useful to use the matrices 
$$\left[ \begin{matrix} \Y_{\tau}^{-1/2} & 0 \\ 0 & \Y_{\tau}^{1/2} \end{matrix}\right] \mathrm{~and~} 
\left[ \begin{matrix} \Y_{\tau}^{1/2} & 0 \\ 0 & \Y_{\tau}^{-1/2} \end{matrix}\right]$$ instead for a step in the clockwise (resp. counterclockwise) direction.  
In other words, we divide by $\sqrt{\Y_{\tau}}$ for every 
elementary step of type $2$ crossing the arc $\tau$.  If so, we will let $\overline{M}(\rho_{\gamma}) =
\overline{M}(\rho_t) \cdots \overline{M}(\rho_2) \overline{M}(\rho_1)$ denote the corresponding matrix product in $SL_2(\R)$.
\end{definition}

\begin{remark}
Clearly $M(\rho_{\gamma})$ and $\overline{M}(\rho_{\gamma})$ depend on our choice of $\rho_{\gamma}$.  
However, it turns out that nevertheless,
we can read off from them invariants which depends only on $\gamma$.
\end{remark}

\begin{definition}
Given a $2 \times 2$ matrix $M=(m_{ij})$, let $\UR(M)$ denote $m_{12}$.  Let $\tr(M)$ denote the trace of $M$.
\end{definition}

\begin{lemma} \label{Lem MA}
Fix $(S,M)$ and $T$  as usual.  Let
$\gamma_1$ and $\gamma_2$ be a generalized arc and 
a closed loop, respectively, with no contractible kinks.
Then for any two $M$-paths $\rho$ and $\rho'$ associated to $\gamma_1$, 
we have that 
$$|\UR(M(\rho)| = |\UR(M(\rho'))|.$$
And for any two $M$-paths $\rho$ and $\rho'$ associated to $\gamma_2$, 
we have that 
$$|\tr(M(\rho)| = |\tr(M(\rho'))|.$$
The analogous results also hold for the reduced matrices corresponding to an $M$-path.
\end{lemma}

Lemma \ref{Lem MA} allows us to make the following definitions.

\begin{definition}\label{def:chi}
Let $\gamma$ be a generalized arc and $\gamma'$ be a closed loop,
and let $\rho$ and $\rho'$ denote arbitrary $M$-paths associated to $\gamma$
and $\gamma'$, respectively.
We associate three (related) algebraic quantities to $\gamma$ and $\gamma'$:
\begin{enumerate}
\item $\hat{\chi}_{\gamma,T} = |\UR(M(\rho))|$ and
$\hat{\chi}_{\gamma',T}=|\tr(M(\rho'))|$.
\item $\overline{\chi}_{\gamma,T}=|\UR(\overline{M}(\rho))|$ and 
$\overline{\chi}_{\gamma',T}=|\tr(\overline{M}(\rho'))|$.
\item $\chi_{\gamma,T} = \Phi(\hat{\chi}_{\gamma,T})$ and 
   $\chi_{\gamma',T} = \Phi(\hat{\chi}_{\gamma',T}).$ 
\end{enumerate}
\end{definition}
When $T$ has no self-folded triangles, the first and third definitions
coincide.  The third definition is the most fundamental, and we will 
show in Section \ref{Sec:matrix=match} that $\chi_{\gamma,T} = X_{\gamma,T}$,
where $X_{\gamma,T}$ is the sum over matchings which we defined in 
Section \ref{sect main}.  The second definition will be used in our
proofs of skein relations in Section \ref{sec:skein}.

\begin{proof} [Proof of Lemma \ref{Lem MA}]
Consider an $M$-path $\rho$. 
First note that the two $M$-paths $\rho_1$ and $\rho_2$ in Figure
\ref{fig infinitesimal} have the property that 
$M(\rho_1) = -M(\rho_2)$; in other words, they 
are equal as elements of $PSL_2$.  This means that a local adjustment
of an $M$-path which replaces one segment $\rho_1$ by another segment
$\rho_2$ will not affect the value of $|\UR(M(\rho))|$ or 
$|\tr(M(\rho))|$.  

Next, note that 
$\left[ \begin{matrix} 0 & x_{\tau} \\ 
   \frac{-1}{x_{\tau}} & 0 \end{matrix}\right]
\left[ \begin{matrix} \Y_{\tau} & 0 \\ 
   0 & 1 \end{matrix}\right] = 
\left[ \begin{matrix} 1 & 0 \\ 
   0 & \Y_{\tau} \end{matrix}\right]
\left[ \begin{matrix} 0 & x_{\tau} \\ 
   \frac{-1}{x_{\tau}} & 0 \end{matrix}\right].$
Dividing by $\sqrt{\Y_{\tau}}$ on both sides also preserves this identity.  
This implies that if at some point an $M$-path 
crosses an arc and then travels along that arc, then 
a local adjustment which switches the order of these two
steps will not affect the final value of $|\UR(M(\rho))|$
or $|\tr(M(\rho))|$.  

Finally, if $\rho$ is the $M$-path for a generalized arc $\gamma$,
note that we can change its starting point $P_s=v_{s,\tau}^{\pm}$ 
by choosing the other point
$v_{s,\tau'}^{\pm}$ which is obtained from $P_s$ by traveling around the circle
$h_s$ without crossing any arcs of $T$.  The result 
will still be an $M$-path $\rho'$
for $\gamma$, and we will have that 
$M(\rho') = M(\rho) N$ where $N$ is lower-triangular with $1$'s on the diagonal.
Therefore $|\UR(M(\rho'))| = |\UR(M(\rho))|$.  Similarly, 
we can change the ending 
point $P_t$, and we will still have an $M$-path $\rho'$ for $\gamma$ such that 
$|\UR(M(\rho'))| = |\UR(M(\rho))|$.

Now the proof of Lemma \ref{Lem MA} for generalized arcs follows from the fact
that any two $M$-paths for $\gamma$ can be obtained from each other by 
a combination of the above local adjustments.
To complete the proof in the case that $\gamma$ is a closed loop, note that 
choosing a different starting point $P_s=P_t$ for the $M$-path $\rho$ 
amounts to circularly re-arranging the matrix product.  
The fact that this operation does not affect $|\tr(M(\rho))|$ 
follows from the identity $\tr(UV) = \tr(VU)$.  The analogous identities for 
$|\UR(\overline{M}(\rho))|$ and $|\tr(\overline{M}(\rho))|$ also hold.
\end{proof}

\begin{remark}
\label{ReducedMats} We can actually strengthen 
Lemma \ref{Lem MA} in the case of reduced matrices.
Whenever $\rho$ is an $M$-path (for some generalized arc or closed loop) from $s$ to $t$ and $\rho'$ is another path from $s$ to $t$ that is isotopic to 
$\rho$ (with possibly some extra intersections with $T$), we have that
$\overline{M}(\rho) = \pm \overline{M}(\rho')$.
\end{remark}

Our first main result is the following.

\begin{theorem} \label{thm matrix-match}
Let $(S,M)$ be a bordered surface with an ideal triangulation
$T$, and
let $\iota(T)=\{\tau_1,\tau_2,\dots, \tau_n\}$ be the
corresponding tagged triangulation.
Let $\mathcal{A}$ be
the corresponding cluster algebra with principal coefficients with respect to  $\Sigma_T=(\mathbf{x}_T,\mathbf{y}_T,B_T)$.  
\begin{itemize}
\item Suppose $\zg$
is a generalized arc in $S$ without contractible kinks (this may include a noose).
Let $G_{T,\zg}$ be the graph constructed in Section \ref{sect graph}.  Then 
$$\chi_{\gamma,T} = 
\frac{1}{\mathrm{cross}(T,\zg)} \sum_P 
x(P) y(P),$$ where the sum is over all perfect matchings $P$ of $G_{T,\zg}$.
Combining this with Theorem \ref{thm MSW}, 
it follows that when $\gamma$ is an arc, $\chi_{\gamma,T}$ is the 
Laurent expansion of $x_{\zg}$ with respect to $\Sigma_T$.
\item Suppose that $\zg$ is a closed loop which is not contractible, has no contractible kinks, and does
not enclose a single puncture.  Then 
$$\chi_{\gamma,T} = \frac{1}{\mathrm{cross}(T,\zg)} \sum_{P} 
x(P) y(P),$$ where the sum is over all good matchings $P$ of the  band graph $\widetilde{G}_{T,\zg}$.
\end{itemize}
\end{theorem}

Comparing our formulas to those 
of Fock and Goncharov \cite{FG3}, we observe the following.

\begin{prop} \label{FGXcoord}
Fix $(S,M)$ and $T$.  Let $\gamma$ be a generalized arc or closed loop,
and suppose that it crosses arcs $\tau_{i_1},\dots,\tau_{i_d}$ in $T$.
Then if we substitute $x_i=1$ and $y_i=X_i$ into  
$\overline{\chi}_{\gamma,T}$, the
resulting expression will give
the associated $X$-coordinates for $\gamma$ with respect to $T$.  
On the other hand, if we substitute $x_i=A_i$ and $y_i=1$ into 
$\chi_{\gamma,T}$, the resulting expression will give 
the associated $A$-coordinate for $\gamma$ with respect to $T$.
\end{prop}   

\begin{proof}
In the case of $A$-coordinates, this result is immediate.  Comparing our notation with that of Fock and Goncharov \cite{FG3}, their $F$-matrices realize elementary steps of type $1$ and their $D$-matrices realize elementary steps of type $3$.  (Note that when we set $y_i=1$, a matrix corresponding to an elementary step of type $2$ is the identity matrix.)

To obtain their formula for $X$-coordinates we coarsen our vertex structure on $(S,M)$ by using the subset $V^-$ consisting only of the $v_{m,\tau}^-$'s.  We thereby get a coarsened graph which contains a triangle (as opposed to a hexagon) for each triangle of the triangulation $T$ and a single edge (as opposed to two edges) crossing each $\tau \in T$.  See Figure \ref{fig FG-Xcoord}.  

\begin{figure}
\input{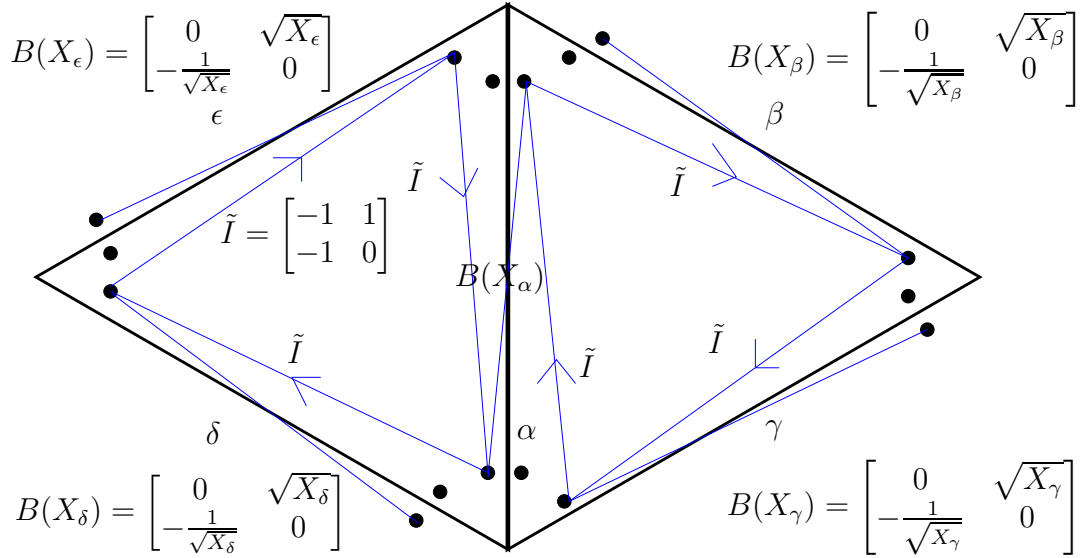}
\caption{A quadrilateral inside triangulation $T$ with steps between vertices of $V^-$ highlighted}
\label{fig FG-Xcoord}
\end{figure}

Any $M$-path from $P_s  = v_{m_1, \tau_1}^-$ to $P_t = v_{m_2, \tau_2}^-$ can be decomposed into quasi-elementary steps, each of which goes between vertices of $V^-$.  One possible quasi-elementary step combines a (counterclockwise) step of type $1$ followed by a step of type $3$.  We obtain

$$ \tilde{I} = \left[\begin{matrix} 
0 & 1 \\ -1 & 0 \end{matrix}\right] 
\left[\begin{matrix} 
1 & 0 \\ -1 & 1 \end{matrix}\right] =
\left [ \begin{matrix} -1 & 1 \\ -1 & 0 \end{matrix}\right ].$$
  
\noindent The other possible quasi-elementary step crosses an arc $\tau\in T$ and combines a step of $2$ and a step of type $3$ (in the positive direction).  These steps correspond to the matrices
$$B(X_\tau) = \frac{1}{\sqrt{X_\tau}}
\cdot \left[ \begin{matrix}
0 & 1 \\ -1 & 0 \end{matrix}\right] 
\left[\begin{matrix} 
1 & 0 \\ 0 & X_\tau \end{matrix}\right] = 
\frac{1}{\sqrt{X_\tau}}
\cdot 
\left[\begin{matrix} 
X_\tau & 0 \\ 0 & 1 \end{matrix}\right]
\left[ \begin{matrix}
0 & 1 \\ -1 & 0 \end{matrix}\right] =  
\left[ \begin{matrix} 0 & \sqrt{X_\tau} \\ -\frac{1}{\sqrt{X_\tau}} & 0 \end{matrix}\right],$$ 
where we have divided by $\sqrt{X_{\tau}}$ as our $M$-path transverses $\tau\in T$, which is crossed by the arc $\gamma$.  

Note that the matrix $\tilde{I}$ differs slightly from the matrix $I = \left [ \begin{matrix} 1 & 1 \\ -1 & 0 \end{matrix}\right ]$ used by Fock and Goncharov.  We will utilize the matrix equality 
$$\left [\begin{matrix} - A_1 & B_1 \\ C_1 & - D_1\end{matrix}\right]
\left [\begin{matrix}  - A_2 &  B_2 \\  C_2 & - D_2\end{matrix}\right] =
- \left [\begin{matrix} -(A_1 A_2 + B_1 C_2) & (A_1 B_2 + B_1 D_2) \\ 
(C_1 A_2 + D_1 C_2) & - (C_1 B_2 + D_1 D_2) \end{matrix}\right],
$$
which indicates how the matrix product changes if we negate the diagonal entries of the constituent matrices.
As our formulas for $\overline{\chi}_{\gamma,T}|_{x_i=1,y_i=X_i}$ only involve matrix products consisting of $\tilde{I}$, $\tilde{I}^2$, and the anti-diagonal matrices $B(X_\tau)$, we have by induction that the absolute values of the upper right entries and traces of these matrix products agree with the formulas for $X$-coordinates in Section 4.1 of \cite{FG3}.
\end{proof}

\vspace{1em}

Theorem \ref{thm matrix-match} immediately implies the following.
\begin{cor}\label{pos}
The quantity $\chi_{\gamma,T}$ is a Laurent polynomial with all coefficients
positive.
\end{cor}

In the case that each $y_i=1$, Corollary \ref{pos} was also proved 
by Fock and Goncharov in 
\cite[Section 12.2]{FG1}.

\section{The matching and matrix-product 
formulas coincide}
\label{Sec:matrix=match}

In this section
we will prove Theorem \ref{thm matrix-match}.
We will start by giving
two general combinatorial 
results in Section \ref{comb-matching} about how one can enumerate
matchings of (abstract) snake and band graphs using appropriate products
of $2 \times 2$ matrices, and then apply these results in the 
case that the snake and band graphs come from arcs and loops in 
a surface.

\subsection{Matchings of abstract snake and band graphs}\label{comb-matching}

\begin{definition}[Abstract snake graph]\label{abstractsnake}
An \emph{abstract snake graph} with $d$ tiles is formed by concatenating
the following \emph{puzzle pieces}:
\begin{itemize}
\item An \emph{initial triangle} \quad \raisebox{-5pt}{\includegraphics[height=.35in]{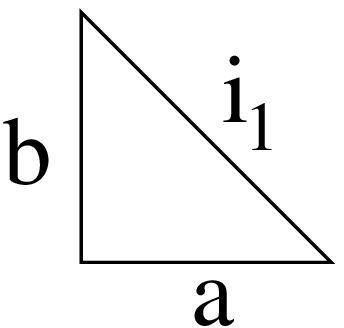}}
\item $d-1$ \emph{parallelograms} $H_1,\dots, H_{d-1}$, where 
each $H_j$ is either 
\quad \raisebox{-5pt}{\includegraphics[height=.55in]{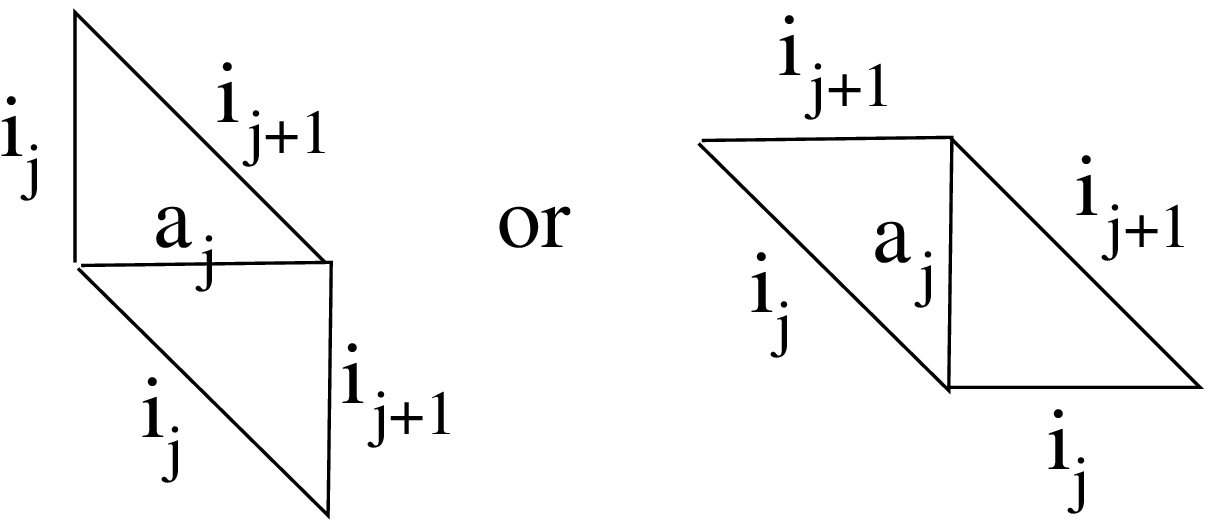}}

\noindent (a \emph{north-pointing} or \emph{east-pointing} parallelogram)
\item A \emph{final triangle} \quad \raisebox{-5pt}{\includegraphics[height=.35in]{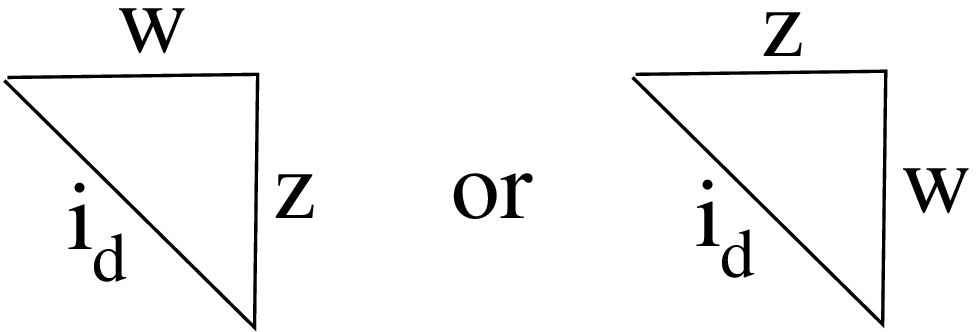}},  based on whether $d$ is odd or even.
\end{itemize}
We then erase all diagonal edges (those with slope $-1$)
from the figure.
\end{definition}

\begin{definition}[Abstract band graph]\label{abstractband}
An \emph{abstract band graph} with $d$ tiles is formed by concatenating
the following \emph{puzzle pieces}:
\begin{itemize}
\item An \emph{initial triangle} \quad \raisebox{-5pt}{\includegraphics[height=.35in]{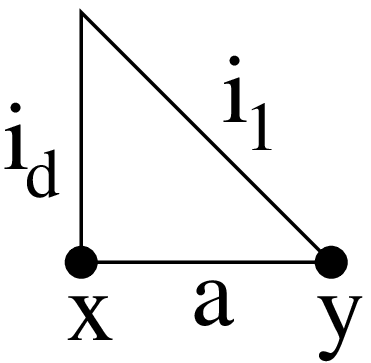}}
\item $d-1$ \emph{parallelograms} $H_1,\dots, H_{d-1}$, where 
each $H_j$ is as before.
\item A \emph{final triangle} \quad \raisebox{-5pt}{\includegraphics[height=.4in]{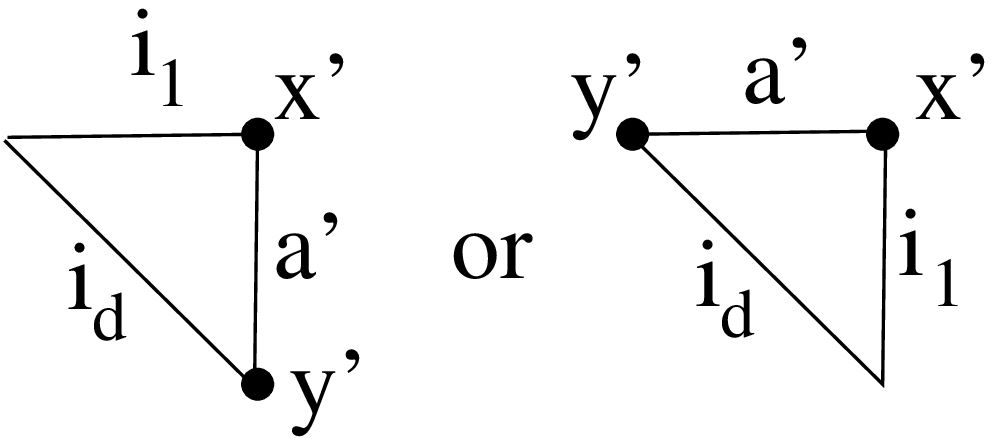}}, based on whether $d$ is odd or even.
\end{itemize}
We then identify the edges $a$ and $a'$, 
the vertices $x$ and $x'$, and 
the vertices $y$ and $y'$. 
Finally, we erase all diagonal edges (those with slope $-1$)
from the figure.
\end{definition}

Just as in Definitions \ref{def:perfect} and \ref{def good matching}, 
we can consider
the perfect matchings of an abstract snake graph 
and the good matchings of an abstract band graph.  Additionally,
we can use Definitions \ref{def:perfect} and \ref{height} to associate to 
each perfect matching $P$ of an abstract snake graph
and to each good matching $P$ of an abstract band graph 
its \emph{weight} and \emph{height} monomials
$x(P)$ and $h(P)$. 

\begin{definition}\label{abstractmatrix}
Let $G$ be an abstract snake or band graph with $d$
tiles.  We will associate to $G$ a matrix 
$M_d$.
First we define some matrices $m_1$,\dots,$m_{d-1}$,
where each $m_i$ is either
$\left[\begin{matrix} 1 & 0 \\ \frac{x_{a_j}}{x_{i_j}x_{i_{j+1}}} & \Y_{i_j} \end{matrix}\right]$
or 
$\left[\begin{matrix} \frac{x_{i_{j+1}}}{x_{i_j}} & x_{a_j} \Y_{i_j} \\ 0 &  \frac{x_{i_j} \Y_{i_j}}{x_{i_{j+1}}} \end{matrix}\right]$,
subject to the following conditions:
\begin{itemize}
\item $m_1$ is of the first type if $H_1$ is a north-pointing
parallelogram, and otherwise it is of the second type;
\item for $i>1$, $m_i$ is of the first type if both 
$H_{i-1}$ and $H_{i}$ have the same shape, and otherwise,
it is of the second type.
\end{itemize}
Finally, if $d=1$, we set $M_1 = 
\left[\begin{matrix} 1 & 0 \\ 0 & 1 \end{matrix}\right],$ and
otherwise, we set $M_d = m_{d-1} \dots m_1$.
\end{definition}

A main result of this section is the following.
\begin{theorem}\label{match1}
Suppose that $G$ is an abstract snake graph with $d$ tiles.  Then its perfect matching enumerator
is given by 
$$\sum_P x(P) h(P) = x_{i_1} \dots x_{i_d} \UR \bigg( \left[\begin{matrix} \frac{x_w}{x_{i_d}} & x_z \Y_{i_d} \\ -\frac{1}{x_z} & 0 \end{matrix}\right] 
M_d 
\left[\begin{matrix} 0 & x_a \\ -\frac{1}{x_a} & \frac{x_b}{x_{i_1}}\end{matrix}\right] \bigg),$$
where the sum is over all perfect matchings of $G$.

Now suppose that $G$ is an abstract band graph with $d$ tiles.  Then its good matching enumerator
is given by 
$$\sum_P x(P) h(P) = x_{i_1} \dots x_{i_d} \tr \bigg( 
\left[\begin{matrix} \frac{x_{i_1}}{x_{i_d}} & x_a \Y_{i_d} \\ 0 & \frac{\Y_{i_d} x_{i_d}}{x_{i_1}}\end{matrix}\right] M_d \bigg),$$ where the sum is over all good matchings of $G$.
\end{theorem}

The main step towards proving Theorem \ref{match1} is the following.
\begin{prop} \label{prop:AbstractMatrixElts}
Let $G$ be an abstract snake graph with $d$ tiles.  
Write $M_d = \left[\begin{matrix} A_d & B_d \\ C_d & D_d \end{matrix} \right]$.
Then we have
\begin{eqnarray*}
A_d &=& \frac{\sum_{P\in S_A} x(P) h(P)}
{(x_{i_1} \cdots x_{i_{d-1}}) x_a x_{w}} \qquad \qquad
B_d = \frac{\sum_{P\in S_B} x(P)h(P)}
{(x_{i_2} \cdots x_{i_{d-1}}) x_b x_{w}} \\
C_d &=& \frac{\sum_{P\in S_C} x(P)h(P)}
{(x_{i_1} \cdots x_{i_d}) x_a x_z \Y_{i_d}} \qquad \qquad
D_d = \frac{\sum_{P\in S_D} x(P)h(P)}
{(x_{i_2} \cdots x_{i_d}) x_b x_z \Y_{i_d}},
\end{eqnarray*}
where 
$S_A$, $S_B$, $S_C$, and $S_D$ are the sets of perfect matchings of $G$ which 
use the edges $\{a,w\}$; $\{b,w\}$; $\{a,z\}$; and $\{b,z\}$,
respectively.
\end{prop}

\begin{proof}
The proof of this proposition is straightforward, using induction
on $d$, and considering what happens as one adds one more tile to a snake graph.  When $d=1$, the graph $G$ consists 
of an initial triangle glued to a final triangle, set $S_A = \{\mathrm{minimal~matching~of~}G\}$, sets $S_B$, $S_C$ are empty, and set $S_D = \{\mathrm{maximal~matching~of~}G\}$.  Thus the base case, where $M_1$ equals the $2$-by-$2$ identity matrix, holds.  

For $d>1$, we let $G'$ denote the graph obtained by gluing together the initial triangle, parallelograms $H_1, \dots, H_{d-2}$, and the final triangle.  For convenience, we label the final triangle in $G'$ with $w'$ and $z'$, and note that the orientation of this triangle depends on whether $(d-1)$ is odd or even.  By changing edge labels, we observe that the graph $G'$ is isomorphic to the subgraph of $G$ consisting of the first $(d-1)$ tiles.  In particular, we either replace the edge label $w'$ with $a_{d-1}$ and $z'$ with $i_d$, or vice-versa.  
In the first case, we have bijections between the following pairs  
of perfect matchings: 
\begin{itemize}
\item $S_A(G) \leftrightarrow S_A(G')$, 
\item $S_B(G)\leftrightarrow S_B(G')$, 
\item $S_C(G) \leftrightarrow S_A(G') \sqcup S_C(G')$, and 
\item $S_D(G) \leftrightarrow S_B(G') \sqcup S_D(G')$.  
\end{itemize}
See Figures \ref{FigGG} and \ref{GG-Matrix}.  In the second case, we have bijections between the following
pairs of perfect matchings:
\begin{itemize}
\item $S_A(G) \leftrightarrow S_A(G') \sqcup S_C(G')$, 
\item $S_B(G) \leftrightarrow S_B(G') \sqcup S_D(G')$,
\item $S_C(G) \leftrightarrow S_C(G')$, and 
\item $S_D(G) \leftrightarrow S_D(G')$.
\end{itemize}
Note that the set $S_A$ contains the minimal matching of $G$, while $S_D$ contains the maximal matching.  Consequently, by altering the weights and heights accordingly, we obtain
\begin{equation}
\label{MatrixCase1}
\left[\begin{matrix} A_{d} & B_{d} \\ C_{d} & D_{d} \end{matrix}\right] 
= \left[\begin{matrix} 1 & 0 \\ \frac{x_{a_{d-1}}}{x_{i_{d-1}}x_{i_{d}}} & \Y_{i_{d-1}} \end{matrix}\right] 
\left[\begin{matrix} A_{d-1} & B_{d-1} \\ C_{d-1} & D_{d-1} \end{matrix}\right] 
\end{equation}
in the first case, and we obtain
\begin{equation}
\label{MatrixCase2}
\left[\begin{matrix} A_{d} & B_{d} \\ C_{d} & D_{d} \end{matrix}\right] 
= \left[\begin{matrix} \frac{x_{i_{d}}}{x_{i_{d-1}}} & x_{a_{d-1}} \Y_{i_{d-1}} \\ 0 & \frac{x_{i_{d-1}} \Y_{i_{d-1}}}{x_{i_{d}}} \end{matrix}\right] \left[\begin{matrix} A_{d-1} & B_{d-1} \\ C_{d-1} & D_{d-1} \end{matrix}\right] 
\end{equation}
in the second case.  Comparing these equations with the definition of matrix $m_{d-1}$, we see that the two cases agree with the two cases in Definition \ref{abstractmatrix}.
\end{proof}

\begin{figure}
\input{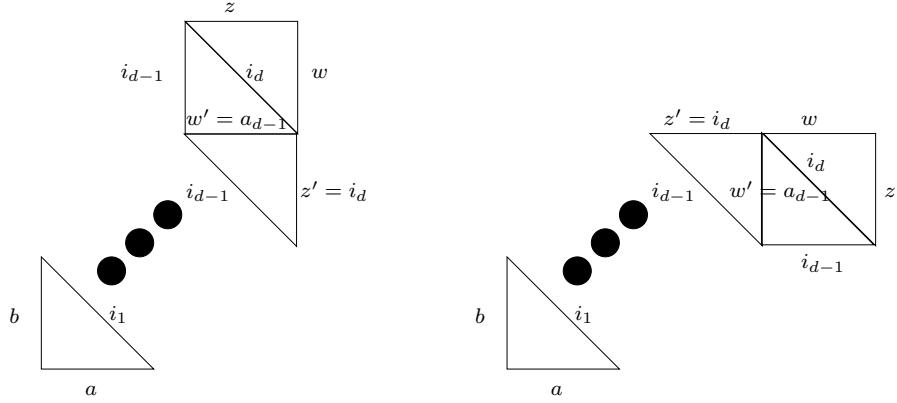}
\caption{The graph $G'$ as a subgraph of $G$ in the case $w'=a_{d-1}$ and $z'=i_d$.  (Left): d is even. (Right): d is odd}
\label{FigGG}
\end{figure}

\begin{figure}
\input{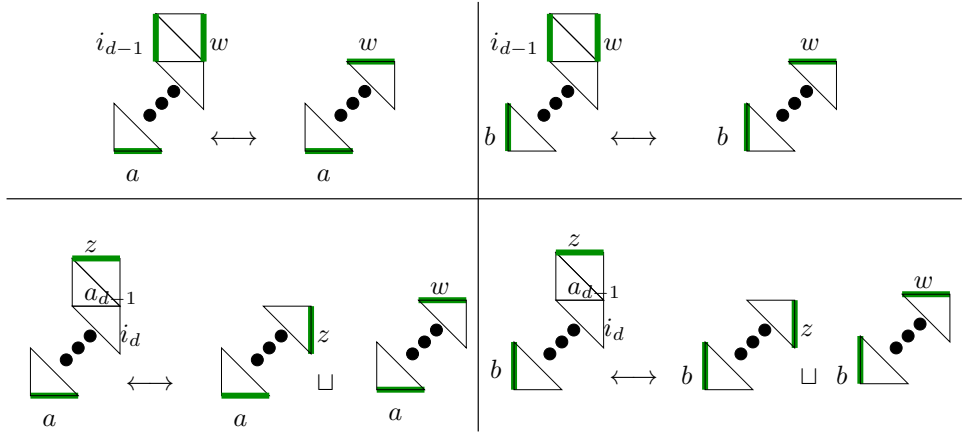}
\caption{Decomposing matrix $M_d$ in the first case and when $d$ is even}
\label{GG-Matrix}
\end{figure}

We have the following immediate corollary of Proposition 
\ref{prop:AbstractMatrixElts}.
\begin{corollary}\label{cor:permatch}
Let $G$ be an abstract snake graph with $d$ tiles.  
Write $M_d = \left[\begin{matrix} A_d & B_d \\ C_d & D_d \end{matrix} \right]$.
Then 
\begin{align}
\frac{\sum_P x(P)h(P)}{x_{i_1}\dots x_{i_d}} &=
\frac{x_a x_w A_d}{x_{i_d}} + \frac{x_b x_w B_d}{x_{i_1} x_{i_d}} + 
\Y_{i_d} x_a x_z C_d + \frac{\Y_{i_d} x_b x_w D_d}{x_{i_1}}\\
\label{matrixformula}
&= 
\UR \bigg( \left[\begin{matrix} \frac{x_w}{x_{i_d}} & x_z \Y_{i_d} \\ -\frac{1}{x_z} & 0 \end{matrix}\right] 
 \left[\begin{matrix} A_d & B_d \\ C_d & D_d \end{matrix}\right]
\left[\begin{matrix} 0 & x_a \\ -\frac{1}{x_a} & \frac{x_b}{x_{i_1}}\end{matrix}\right] \bigg),
\end{align}
where the sum at the left is over all perfect matchings of $G$.
\end{corollary}

Proposition \ref{prop:AbstractMatrixElts} 
also implies the following.
\begin{prop} \label{prop:AbstractBandElts}
Let $G$ be an abstract snake graph with $d$ tiles, but 
with a labeling obtained by substituting 
$i_1$ for $w$, $a'$ for $z$, and $i_d$ for $b$
(the same labeling which is used for a band graph).
Write $M_d = \left[\begin{matrix} A_d & B_d \\ C_d & D_d \end{matrix} \right]$.
Then we have
\begin{eqnarray*}
A_d &=& \frac{\sum_{P\in S_A} x(P) h(P)}
{(x_{i_1} \cdots x_{i_{d-1}}) x_a x_{i_1}} \qquad \qquad
B_d = \frac{\sum_{P\in S_B} x(P)h(P)}
{(x_{i_2} \cdots x_{i_{d-1}}) x_{i_d} x_{i_1}} \\
C_d &=& \frac{\sum_{P\in S_C} x(P)h(P)}
{(x_{i_1} \cdots x_{i_d}) x_a x_{a'} \Y_{i_d}} \qquad \qquad
D_d = \frac{\sum_{P\in S_D} x(P)h(P)}
{(x_{i_2} \cdots x_{i_d}) x_{i_d} x_{a'} \Y_{i_d}},
\end{eqnarray*}
where 
$S_A$, $S_B$, $S_C$, and $S_D$ are the sets of perfect matchings of $G$ which 
respectively use the edges $a$ and $i_1$
from the first and last triangle, 
$i_d$ and $i_1$ from the first and last triangle, 
$a$ and $a'$ from the first and last triangle, and 
$i_d$ and $a'$ from the first and last triangle.
\end{prop}

\begin{corollary}\label{band-formula}
Let $G$ be an abstract band graph with $d$ tiles.  
Write $M_d = \left[\begin{matrix} A_d & B_d \\ C_d & D_d \end{matrix} \right]$.
Then 
\begin{align*}
\frac{\sum_P x(P)h(P)}{x_{i_1}\dots x_{i_d}} &=
\frac{x_{i_1} A_d}{x_{i_d}} +  
\Y_{i_d} x_a  C_d + \frac{\Y_{i_d} x_{i_d} D_d}{x_{i_1}}\\
&= 
\tr \bigg( 
\left[\begin{matrix} \frac{x_{i_1}}{x_{i_d}} & x_a \Y_{i_d} \\ 0 & \frac{\Y_{i_d} x_{i_d}}{x_{i_1}}\end{matrix}\right] 
\left[\begin{matrix} A_d & B_d \\ C_d & D_d \end{matrix}\right]
\bigg),
\end{align*}
where the sum at the left is over all good  matchings of $G$.
\end{corollary}
Note that Corollary \ref{cor:permatch} and Corollary \ref{band-formula}
immediately imply the first and second parts of Theorem \ref{match1}.

\begin{proof}
Consider the sets of matchings
$S_A$, $S_B$, $S_C$ and $S_D$ which
were defined in  Proposition \ref{prop:AbstractBandElts}.
Let $G$ be the snake graph from Proposition \ref{prop:AbstractBandElts},
and let $\widetilde{G}$ denote the band graph obtained from $G$
by identifying edge $a$ and $a'$.
Note that every perfect matching of $G$ in $S_A$ (respectively,
$S_C$ and $S_D$) descends to a good
matching of $\widetilde{G}$ if we remove the edge $a$ (respectively,
$a$ and $a'$) from it.
Moreover, every good matching of $\widetilde{G}$ can be obtained
uniquely from one of the sets $S_A$, $S_C$, and $S_D$.
(On the other hand, no matching $P$ from $S_B$ can give rise to 
a good matching of $\widetilde{G}$.) This completes the proof.
\end{proof}

\subsection{The standard $M$-path} 

To facilitate the proof of Theorem \ref{thm matrix-match}, 
we will associate to each arc $\gamma$ 
a {\emph standard $M$-path} $\rho_{\gamma}$, and show that the matrix
formula coming from that $M$-path has the same form as  
\eqref{matrixformula} from Corollary \ref{cor:permatch}.

\begin{definition}[Standard $M$-path for an arc]\label{def:standard}
Let $\gamma$ be a generalized arc that goes from point $P$ to point $Q$, 
crossing the arcs $\tau_{i_1},\dots, \tau_{i_d}$ in order.  
Label the first triangle $\Delta_0$ that $\gamma$ crosses with sides $a$, $b$, and $\tau_{i_1}$ in clockwise order so that $P$ is the intersection of the arcs $a$ and $b$; and label the last triangle $\Delta_d$ crossed with sides $w$, $z$, and $\tau_{i_d}$ in clockwise order, with $Q$ being the intersection of the arcs $w$ and $z$.  See Figure \ref{retracted}.

\begin{figure}
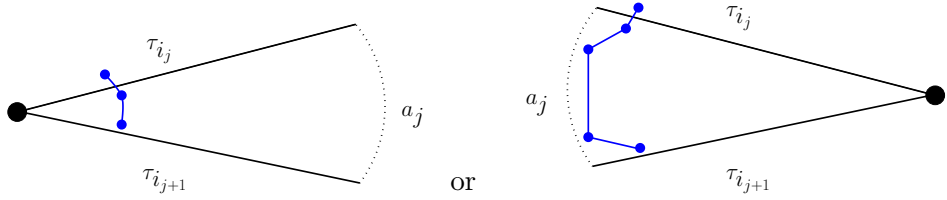

\scalebox{0.35}{\input{Trans1-v2.pstex_t}} \hspace{1em} or \hspace{1em}
\scalebox{0.35}{\input{Trans2-v2.pstex_t}}
\caption{Possible transitions between $\tau_{i_j}$ and $\tau_{i_{j+1}}$ in the standard $M$-path.
}
\label{trans}
\end{figure}

We define $\rho_{\gamma}$ so that it starts and ends
at points $v^{\pm}_{P,a}$ and $v^{\pm}_{Q,z}$,
where the sign is chosen so that the starting and ending points
lie inside the triangles $\Delta_0$ and $\Delta_d$.
The path $\rho_{\gamma}$ starts with an elementary step of type 3 along arc $a$,
followed by a step of type 1 between arcs $a$ and $\tau_{i_1}$.
The segment we have defined so far ends at a point of the form $v^{\pm}_{*, \tau_{i_1}}$
(and does not cross $\tau_{i_1}$).

Subsequently, we define the sequence of elementary steps in $\rho_{\gamma}$ based on whether 
the arc $\tau_{i_{j+1}}$ lies clockwise or counterclockwise from the arc $\tau_{i_j}$ 
in the unique triangle containing the corresponding segment of $\rho_{\gamma}$.
If it is counterclockwise, we proceed with the definition of $\rho_{\gamma}$ by 
adding an elementary step of type 2 which crosses $\tau_{i_j}$ and then 
a step of type 1 
between $\tau_{i_j}$ and $\tau_{i_{j+1}}$.  
If the orientation is clockwise, we 
again begin with a step of type 2 which crosses $\tau_{i_j}$, however, we then have a step of type 1 
between $\tau_{i_j}$ and $a_j$.  We then follow a step of type 3 in the positive direction along $a_j$, succeeded 
by a step of type 1 between $a_j$ and $\tau_{i_{j+1}}$.
In both of these cases, we do not cross or 
touch $\tau_{i_{j+1}}$. 
This progression keeps the path in the same relative position after each double or quadruple
step, and as a consequence, we can iterate our construction.  
See Figure \ref{trans}.  

After $(d-1)$ transitions as in Figure \ref{trans}, the path is at a point $v^{\pm}_{*,\tau_{i_d}}$, on the side closer to the arc labeled $z$,
and is about to cross $\tau_{i_d}$.  
We then add an elementary  step of type 2 which crosses $\tau_{i_d}$,  a step of type 1, and then
a step of type 3 which travels along $z$ to the point $v^{\pm}_{Q,z}$.
We call this particular $M$-path $\rho_{\gamma}$ the \emph{standard $M$-path associated to $\gamma$}.
\end{definition}

\begin{remark}
We remark that in the above definition, if there are three arcs $\tau_{i_j}$, $\tau_{i_{j+1}}$,
and $\tau_{i_{j+2}}$ such that 
$\tau_{i_{j+1}}$ is in the clockwise direction from  $\tau_{i_j}$ and 
$\tau_{i_{j+2}}$ is in the clockwise direction from  $\tau_{i_{j+1}}$, then the 
standard $M$-path $\rho_{\gamma}$ will have some back-tracking: there will be two consecutive 
steps of type 3 which travel in opposite directions  along
$\tau_{i_{j+1}}$. 
\end{remark}

\begin{figure}
\scalebox{0.4}{\input{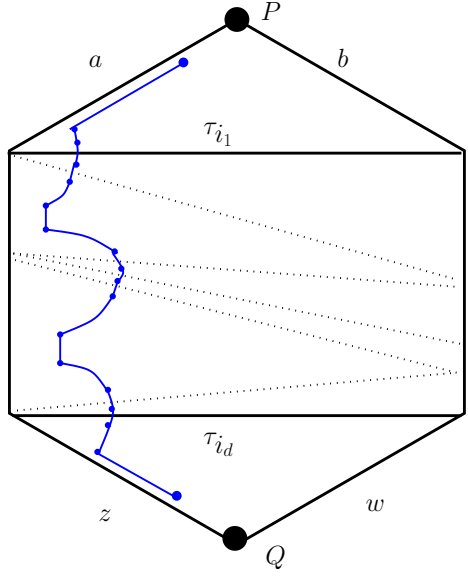}}
\caption{The standard $M$-path $\rho_{\gamma}$ of an arc $\gamma$.
}
\label{retracted}
\end{figure}

\begin{definition}[Standard $M$-path for a closed loop]
Let $\gamma$ be a closed loop which crosses exactly $d$ arcs of $T$ (counted with multiplicity).
Choose a triangle $\Delta$ in $T$ such that two of its arcs are crossed by $\gamma$.
Label those two arcs $\tau_{i_1}$ and $\tau_{i_d}$, where 
$\tau_{i_1}$ is in the clockwise direction from $\tau_{i_d}$.  
Label the third side of $\Delta$ by $a$.  Let $p$ be a point on 
$\gamma$ which lies in $\Delta$ and 
has the form $v^{\pm}_{*,\tau_{i_1}}$.  Finally, let 
$\tau_{i_1},\dots, \tau_{i_d}$ denote the ordered sequence of arcs 
which are crossed by $\gamma$, when one travels from $p$ away from $\Delta$.
We define the \emph{standard $M$-path $\rho_{\gamma}$ associated to $\gamma$} 
exactly as in Definition \ref{def:standard},  starting and ending 
at the point $p$
and travelling along elementary steps based on whether $\tau_{i_{j+1}}$
is counterclockwise or clockwise from $\tau_{i_j}$.  In this case,
we need to consider indices modulo $n$: note that the last elementary 
steps of $\rho_{\gamma}$ will be determined by the fact that 
the arc $i_1$ is in the clockwise direction from $i_d$.
See Figure \ref{retracted2}.
\end{definition}

\begin{figure}
\scalebox{0.4}{\input{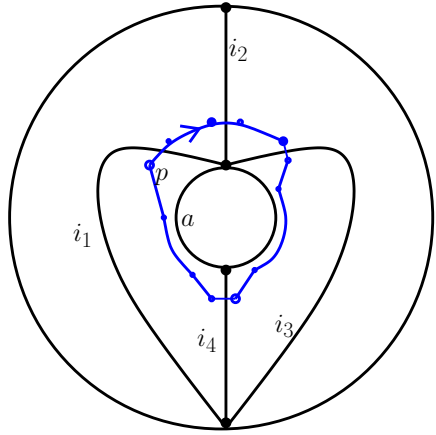}}
\caption{The standard $M$-path $\rho_{\gamma}$ of a loop $\gamma$,
which crosses
$d=4$ arcs of $T$
}
\label{retracted2}
\end{figure}

We now turn to the proof of Theorem \ref{thm matrix-match}.
\begin{proof}
First we consider the case that $\gamma$ is a generalized arc.  Consider its 
standard $M$-path $\rho_{\gamma}$, and recall Definition \ref{matrixforarc}, which 
gives an algorithm for associating a product of matrices to a concatenation of 
elementary steps such as $\rho_{\gamma}$.  Note that 
the first two steps of this path correspond to the matrix product $$\left[\begin{matrix} 1 & 0 \\ \frac{x_b}{x_a x_{i_1}} & 1 \end{matrix}\right]\left[\begin{matrix} 0 & x_a \\ -\frac{1}{x_a} & 0 \end{matrix}\right] = \left[\begin{matrix} 0 & x_a \\ -\frac{1}{x_a} & \frac{x_b}{x_{i_1}}\end{matrix}\right],$$

\noindent and the last three steps of $\rho_{\gamma}$ correspond to $$\left[\begin{matrix} 0 & x_z \\ -\frac{1}{x_z} & 0 \end{matrix}\right] 
\left[\begin{matrix} 1 & 0 \\ \frac{x_w}{x_{i_d} x_z} & 1 \end{matrix}\right]
\left[\begin{matrix} 1 & 0 \\ 0 & \Y_{i_d} \end{matrix}\right] = \left[\begin{matrix} \frac{x_w}{x_{i_d}} & x_z \Y_{i_d} \\ -\frac{1}{x_z} & 0 \end{matrix}\right].$$

In between, the matrix for the portion between $\tau_{i_j}$ and $\tau_{i_{j+1}}$ (for $1 \leq j \leq d-1$) corresponds respectively to 
\begin{eqnarray} \label{trans-matrices} 
&&\left[\begin{matrix} 1 & 0 \\ \frac{x_{a_j}}{x_{i_j}x_{i_{j+1}}} & 1 \end{matrix}\right]
\left[\begin{matrix} 1 & 0 \\ 0 & \Y_{i_j} \end{matrix}\right] = 
\left[\begin{matrix} 1 & 0 \\ \frac{x_{a_j}}{x_{i_j}x_{i_{j+1}}} & \Y_{i_j} \end{matrix}\right]
\mathrm{~~~~or~~~~}  \\
%
\label{trans-matrices2}
&&\left[\begin{matrix} 1 & 0 \\ \frac{x_{i_j}}{x_{a_j}x_{i_{j+1}}} & 1 \end{matrix}\right]
\left[\begin{matrix} 0 & x_{a_j} \\ -\frac{1}{x_{a_j}} & 0 \end{matrix}\right] 
\left[\begin{matrix} 1 & 0 \\ \frac{x_{i_{j+1}}}{x_{a_j}x_{i_{j}}} & 1 \end{matrix}\right]
\left[\begin{matrix} 1 & 0 \\ 0 & \Y_{i_j} \end{matrix}\right] 
= \left[\begin{matrix} \frac{x_{i_{j+1}}}{x_{i_j}} & x_{a_j} \Y_{i_j} \\ 0 & \frac{x_{i_j} \Y_{i_j}}{x_{i_{j+1}}} \end{matrix}\right],
\end{eqnarray}
depending on whether $\tau_{i_{j+1}}$ lies counterclockwise or clockwise from $\tau_{i_j}$.  In both cases, we let $a_j$ be the label of the third side in the triangle bounded by $\tau_{i_j}$ and $\tau_{i_{j+1}}$ which inscribes the appropriate part of the arc $\gamma$.

Applying Definition \ref{matrixforarc} to $\rho_{\gamma}$,   
we find that 
\begin{equation} \label{TripMatrix} 
\MatrixA_{\gamma,T} = \UR \bigg(\left[\begin{matrix} \frac{x_w}{x_{i_d}} & x_z \Y_{i_d} \\ -\frac{1}{x_z} & 0 \end{matrix}\right] 
\left[\begin{matrix} A & B \\ C & D \end{matrix}\right]
\left[\begin{matrix} 0 & x_a \\ -\frac{1}{x_a} & \frac{x_b}{x_{i_1}}\end{matrix}\right]\bigg),\end{equation}
where the middle matrix is obtained by multiplying together a sequence of matrices of the form
$$\left[\begin{matrix} 1 & 0 \\ \frac{x_{a_j}}{x_{i_j}x_{i_{j+1}}} & \Y_{i_j} \end{matrix}\right] \mathrm{~~or~~} \left[\begin{matrix} \frac{x_{i_{j+1}}}{x_{i_j}} & x_{a_j} \Y_{i_j} \\ 0 & \frac{x_{i_j} \Y_{i_j}}{x_{i_{j+1}}} \end{matrix}\right],$$ for $1 \leq j \leq d-1$.

Note that this has precisely the same form as \eqref{matrixformula}, 
and therefore by Corollary \ref{cor:permatch},
$\MatrixA_{\gamma,T}$ has an interpretation in terms of the perfect matchings
of some abstract snake graph $G$.
Moreover, if one compares Definitions 
\ref{abstractsnake} and \ref{abstractmatrix} with the construction of 
snake graphs associated to arcs in Section \ref{sect graph}, it is clear that 
the abstract snake graph $G$ is precisely the snake graph 
$G_{T,\gamma}$ associated to the arc $\gamma$.  This completes the proof of 
Theorem \ref{thm matrix-match} when $\gamma$ is a generalized arc.

Now we consider the case that $\gamma$ is a closed loop.  Consider its 
standard $M$-path $\rho_{\gamma}$.  Note that as before, for $j$ from $1$ to $d-1$,
the matrix for the portion of the path from $\tau_{i_j}$ to $\tau_{i_{j+1}}$  
corresponds respectively to \eqref{trans-matrices} or \eqref{trans-matrices2},
depending on whether $\tau_{i_{j+1}}$ lies counterclockwise or clockwise from $\tau_{i_j}$.  
Since $\tau_{i_1}$ is in the clockwise direction from $\tau_{i_d}$ (by construction), 
the last few steps of 
$\rho_{\gamma}$ correspond to 
$\left[\begin{matrix} \frac{x_{i_1}}{x_{i_d}} & x_a \Y_{i_d} \\ 0 & \frac{\Y_{i_d} x_{i_d}}{x_{i_1}}\end{matrix}\right].$ 

Therefore from Definition \ref{matrixforarc} we have that  
\begin{equation} \label{TripMatrix2} 
\MatrixA_{\gamma,T} = 
\tr \bigg( 
\left[\begin{matrix} \frac{x_{i_1}}{x_{i_d}} & x_a \Y_{i_d} \\ 0 & \frac{\Y_{i_d} x_{i_d}}{x_{i_1}}\end{matrix}\right] 
\left[\begin{matrix} A & B \\ C & D \end{matrix}\right]
\bigg),
\end{equation}
where the rightmost matrix is obtained by multiplying together a sequence of matrices of the form
$$\left[\begin{matrix} 1 & 0 \\ \frac{x_{a_j}}{x_{i_j}x_{i_{j+1}}} & \Y_{i_j} \end{matrix}\right] \mathrm{~~or~~} \left[\begin{matrix} \frac{x_{i_{j+1}}}{x_{i_j}} & x_{a_j} \Y_{i_j} \\ 0 & \frac{x_{i_j} \Y_{i_j}}{x_{i_{j+1}}} \end{matrix}\right],$$ for $1 \leq j \leq d-1$.

Note that this has precisely the same form as the expression in Corollary \ref{band-formula},
and therefore it follows that 
$\MatrixA_{\gamma,T}$ has an interpretation in terms of the good matchings
of some abstract band graph $G$.
Moreover, if one compares Definitions 
\ref{abstractband} and \ref{abstractmatrix} with the construction of 
band graphs associated to arcs in Definition \ref{def band}, it is clear that 
the abstract band graph $G$ is precisely the band graph 
$\widetilde{G}_{T,\gamma}$ associated to the arc $\gamma$.  This completes the proof of 
Theorem \ref{thm matrix-match}.  
\end{proof}

\begin{example}
Consider the loop $\gamma$ and ideal triangulation as in Example \ref{ExampleBand}.  Note that the standard $M$-path $\rho_\gamma$ for this $\gamma$ is illustrated in Figure \ref{retracted2}, where the arcs of triangulation are labelled slightly differently.

Using the arc labels from Figure \ref{fig Ex-band}, we see in this case that $\hat{\chi}_{\gamma,T} = |\tr(M(\rho_{\gamma}))|$, where 
\begin{eqnarray*}
M(\rho_{\gamma}) &=& 
\left[ \begin{matrix} \frac{x_{1}}{x_{4}} & b_2 \Y_{4} \\ 0 & \frac{x_{4}\Y_{4}}{x_{1}} \end{matrix}\right] 
\left[ \begin{matrix} \frac{x_{4}}{x_{3}} & b_3 \Y_{3} \\ 0 & \frac{x_{3}\Y_{3}}{x_{4}} \end{matrix}\right]
\left[ \begin{matrix} 1 & 0 \\ \frac{b_4}{x_{2}x_{3}} & \Y_{2} \end{matrix}\right]
\left[ \begin{matrix} 1 & 0 \\ \frac{b_1}{x_{1}x_{2}} & \Y_{1} \end{matrix}\right]
\\ &=&  \left[\begin{matrix}
\frac{x_1^2x_2x_4 + \Y_3x_1^2b_1b_4 + \Y_3\Y_4x_1x_3b_2b_4 + \Y_2\Y_3x_1x_3b_1b_3 + \Y_2\Y_3\Y_4x_3^2b_1b_2}{x_1x_2x_3x_4} &
\frac{\Y_1\Y_2\Y_3(x_1b_3 + \Y_4x_3b_2)}{x_4}\\
\frac{\Y_3\Y_4(x_1b_4 + \Y_2x_3b_1)}{x_1^2x_2} & \frac{\Y_1\Y_2\Y_3\Y_4x_3}{x_1}
\end{matrix}\right] .
\end{eqnarray*}
Computing the trace of this matrix product, we see that $X_\gamma^T$ and $\chi_{\gamma,T} = \hat{\chi}_{\gamma,T}|_{\Y_i = y_i,b_1=1}$ agree.  The last equality follows from the absence of self-folded triangles in the triangulation $T$ and the fact that the $b_i$'s label boundary segments.
\end{example}

\subsection{Signs of $\hat{\chi}_{\gamma,T}$, $\overline{\chi}_{\gamma,T}$, and $\chi_{\gamma,T}$}
\label{sign-chi}

Recall that the signs of three algebraic quantities defined in the last section, Definition \ref{def:chi}, depended on the choice of the $M$-path 
$\rho$ associated to the generalized arc or loop $\gamma$.  

\begin{lemma} \label{standard-sign}
If we use the standard $M$-path, Definition \ref{def:standard}, for the generalized arc or loop $\gamma$ with no contractible kinks, then every coefficient of $\hat{\chi}_{\gamma,T}$, $\overline{\chi}_{\gamma,T}$, and $\chi_{\gamma,T}$ is positive.
\end{lemma}

For the purposes of this lemma, we consider a contractible loop $\gamma$ to contain a \emph{contractible kink}.  In particular, note that in this case that the signs of the three quantities are \emph{negative}: $\hat{\chi}_{\gamma,T} = \overline{\chi}_{\gamma,T} = \chi_{\gamma,T} = \tr\left(\left[\begin{matrix} -1 & 0 \\ 0 & -1\end{matrix}\right]\right) = -2$ for such a loop.

\begin{proof}
Given $\gamma$, which is a generalized arc or loop which has no contractible kinks, let $\rho_{\gamma}$ denote a corresponding standard $M$-path.  By inspection, the matrices $M(\rho_{\gamma})$ and $\overline{M}(\rho_{\gamma})$ can be decomposed into a product of matrices where almost all entries are positive.  The only negative entries appearing are the lower-left entries of the first and third matrices in \eqref{TripMatrix}; these entries do not affect the upper-right entry of this triple product.  In the case when $\gamma$ is a loop, we use \eqref{TripMatrix2} instead, where no negative entries appear.  
\end{proof}

If $\rho$ is a \emph{non-standard} $M$-path for the generalized arc or loop $\gamma$, then $\rho$ can be deformed into a standard $M$-path $\rho_{\gamma}$ by the local adjustments utilized in the proof of Lemma \ref{Lem MA}:

\begin{enumerate}
\item going from $u$ to $v$ in Figure \ref{fig infinitesimal} clockwise instead of counterclockwise,
\item reversing the order of a step of type 2 crossing an arc $\tau_i \in T$ and a step of type 3 moving along the arc $\tau_i$,
\item rotating around an $h_m$ using a combination of steps of types 1 and 2,
\item reselecting the starting or ending point of the $M$-path corresponding to a closed loop.   
\end{enumerate}  

Each local adjustment of type (1) or (2) changes the sign of the corresponding matrix but does not otherwise affect the matrix.  The local adjustments of type (3) and (4) does not affect $\UR$ or $\tr$ of the resulting matrix product.

\section{Skein relations for generalized arcs and closed loops}\label{sec:skein}

 In this section we prove certain {\it skein relations}, which 
 give multiplication formulas for the cluster algebra elements
 corresponding to  generalized arcs and loops, i.e. 
 arcs and loops which are allowed to have self-intersections.
 We work in the setting of principal coefficients.
 To prove these results, we use the matrix formulas from 
 Section \ref{sec matrix}.
 The proofs of skein relations then follow from the matrix identities in 
 Lemma \ref{matrix-identity}.
Throughout this section we fix a marked surface $(S,M)$, an ideal triangulation 
$T=(\tau_1,\dots,\tau_n)$, 
and the cluster algebra $\A=
\Aprin(B_T)$ which has principal
coefficients with respect to $T$.

\begin{definition}[\emph{Smoothing}]\label{def:smoothing} Let $\gamma, \gamma_1$, and $\gamma_2$ be generalized arcs or closed loops such that we have
one of the following two cases: 

\begin{enumerate}
 \item $\gamma_1$ crosses $\gamma_2$ at a point $x$,
  \item $\gamma$ has a self-intersection at a point $x$.
\end{enumerate}

\noindent We define the \emph{smoothing of $C=\{\gamma_1, \gamma_2\}$ or 
$C=\{\gamma\}$ at the point $x$} to be the pair of configurations $C_+$ 
and $C_-$, where 
$C_+$ (respectively, $C_-$) is the same as $C$ except for the local change that replaces the crossing or self-intersection {\Large $\times$} with the pair of segments {\Large $~_\cap^{\cup}$} (resp., {\large $\supset \subset$}).
\end{definition}   

Note that in the case that two generalized arcs cross, each resulting configuration contains two generalized arcs; if two closed loops cross, each resulting configuration contains only a single closed loop (necessarily with self-intersections); and if a generalized arc and a closed loop cross, then each configuration is a single generalized arc.  When a self-intersection is smoothed, one of the two resulting configurations has an extra closed loop.  See Figures \ref{fig:orient}, \ref{skein3figure}, and \ref{skein5figure}.

\begin{remark}Since we consider generalized arcs up to isotopy, we may assume that each intersection  is transverse.  In particular, there are no triple intersections.
\end{remark}

\begin{figure}
\input{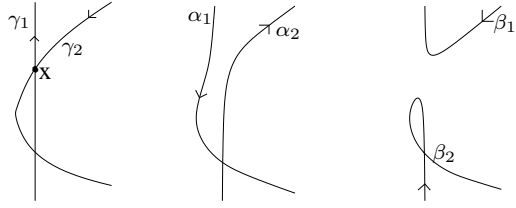}
\caption{Skein relation for two generalized arcs crossing}\label{fig:orient}
\end{figure}

\begin{figure}
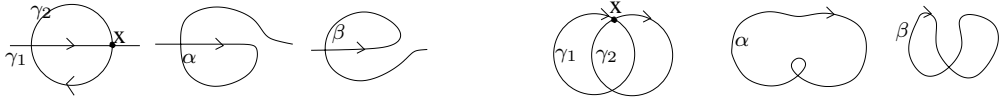

\input{Skein3.pstex_t} \hspace{1.4cm} \input{Skein4.pstex_t}
\caption{Skein relation involving a closed loop}\label{skein3figure}
\end{figure}

\begin{figure}
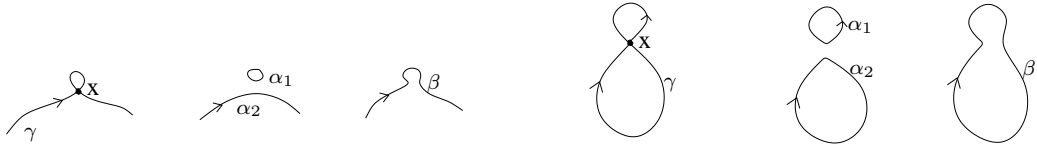

\input{Skein5.pstex_t} \hspace{1.4cm} \input{Skein6.pstex_t}
\caption{Skein relation resolving a self-intersection}\label{skein5figure}
\end{figure}

\subsection{Skein relations for unpunctured surfaces}
When $(S,M)$ is a surface without punctures, there are 
simple formulas for the skein relations.  However, the formulas
are somewhat more complicated to write down in the presence of punctures.  For
this reason, we start by writing down the formula in the unpunctured
case, and give the appropriate generalization (together
with proofs of all results) in Section \ref{withpunctures}.

\begin{Def} \label{elem-laminate}
Let $\gamma$ be an arc in an unpunctured surface $(S,M)$.  Denote by
$L_{\gamma}$ the  curve which 
runs along
$\gamma$ within a small neighborhood of it.  Such a curve is known as the \emph{elementary lamination} associated to $\gamma$, see \cite[Definition 16.2]{FT}.  If $\gamma$
has an endpoint $a$ on a (circular) component $C$ of the boundary of $S$,
then $L_{\gamma}$ begins at a point $a'\in C$ located near $a$ in
the counterclockwise direction, and proceeds along $\gamma$
as shown in Figure \ref{laminate}. 
In particular, if $T=(\tau_1,\dots,\tau_n)$, then we 
let $L_i$ denote $L_{\tau_i}$.
\end{Def}

\begin{figure} \begin{center}
\input{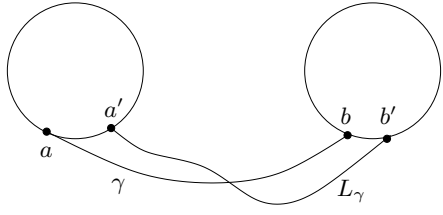}
\end{center}
\caption{Elementary lamination for an arc whose endpoints are on the boundary of $S$}
\label{laminate}
\end{figure}

\begin{prop} \label{up:skein1}
Let $\gamma_1$ and $\gamma_2$ be two generalized arcs
which intersect each other at least once; let $x$ be a point of intersection;
and let $\alpha_1$, $\alpha_2$
and $\beta_1$, $\beta_2$ be the two pairs of arcs obtained by 
smoothing $\gamma_1$
and $\gamma_2$ at $x$.
We then obtain the following identity:
\begin{equation} \label{u:skein-eq1}
\chi_{\gamma_1,T}~ \chi_{\gamma_2,T} = \pm \chi_{\alpha_1,T} ~\chi_{\alpha_2,T}
\prod_{i=1}^n \Y_i^{(c_i-a_i)/2}
 \pm \chi_{\beta_1,T} ~\chi_{\beta_2,T} 
\prod_{i=1}^n \Y_i^{(c_i-b_i)/2},
\end{equation}
where
$c_i=e(\gamma_1,L_i)+e( \gamma_2, L_i)$,
$a_i=e(\alpha_1, L_i)+e(\alpha_2, L_i)$, and
$b_i=e(\beta_1,L_i)+e(\beta_2,L_i)$.
\end{prop}

\begin{prop}
\label{up:skein2}
Let $\gamma_1$ be a generalized arc
or loop and let $\gamma_2$ be a generalized loop, such that
$\gamma_1$ and $\gamma_2$ intersect each other; let $x$ be a point of intersection.
Let $\alpha$ and $\beta$ be the two oriented generalized arcs or loops
obtained by 
{\it smoothing} $\gamma_1$
and $\gamma_2$ at $x$, as in Figure \ref{skein3figure}.  
\begin{equation} \label{uskein-eq2}
\chi_{\gamma_1,T} ~\chi_{\gamma_2,T} =  \pm \chi_{\alpha,T} 
\prod_{i=1}^n \Y_i^{(c_i-a_i)/2}
\pm \chi_{\beta,T}\prod_{i=1}^n \Y_i^{(c_i-b_i)/2},
\end{equation}
where
$c_i=e(\gamma_1,L_i)+e(\gamma_2,L_i)$,
$a_i=e(\alpha, L_i)$, and
$b_i=e(\beta,L_i)$.
\end{prop}

\begin{prop}
\label{up:skein3}
Let $\gamma$ be a generalized arc or closed loop with a self-intersection at $x$.
Let $\alpha_1$, $\alpha_2$, and $\beta$ be the generalized arcs and loops
obtained by 
{\it smoothing} at $x$, as in Figure \ref{skein5figure}.
Then we have the identity
\begin{equation}\label{uskein-eq3}
\chi_{\gamma,T}  = 
\pm \chi_{\alpha_1,T} \chi_{\alpha_2,T} 
\prod_{i=1}^n \Y_i^{(c_i-a_i)/2}
\pm \chi_{\beta,T}\prod_{i=1}^n \Y_i^{(c_i-b_i)/2},
\end{equation}
where
$c_i=e(\gamma_1,L_i)+e(\gamma_2,L_i)$,
$a_i=e(\alpha_1,L_i)+e(\alpha_2,L_i)$, and
$b_i=e(\beta,L_i)$.
\end{prop}

\subsection{Skein relations for surfaces with punctures}\label{withpunctures}
We now give the skein relations for surfaces with punctures.  
\begin{remark}
\label{rem:no-self-folded}
For technical reasons we assume in this section that the ideal triangulation $T$ has no self-folded triangles.  In particular, when there is a self-folded triangle, we cannot use the dictionary between signed intersection numbers and transverse measures that is described below in Remark \ref{FT-dict}. 
\end{remark}
Before stating the relations, we need to define some terminology.
The following definition is a slight variant of the $M$-paths defined in Section \ref{sec matrix}. 

\begin{definition}\label{loosened}
A \emph{loosened $M$-path} 
$\tilde{\rho}_\gamma$ 
for the (oriented) generalized arc $\gamma$ from the marked point $m$ to 
$m'$,
is a concatenation of oriented curves
$\sigma_2\circ \rho_{\gamma} \circ \sigma_1$ 
such that:
\begin{itemize}
\item  $\rho_\gamma$ is an 
$M$-path for $\gamma$
\item $\sigma_1$ (respectively, $\sigma_2$)
is a concatenation of elementary steps of types 1 and 2 
traveling along $h_{m}$ (resp., $h_{m'}$).
\end{itemize}
\end{definition}

\begin{definition} \label{signed-exc}
Let $\tilde{\rho}_\gamma = 
\sigma_2\circ \rho_{\gamma} \circ \sigma_1$ be a loosened $M$-path.  
We define the \emph{signed excess} of $\tilde{\rho}_\gamma$ with respect to the
arc $\tau_i \in T$ as
\begin{eqnarray*}\tilde{\ell}(\tilde{\rho}_\gamma, \tau_i) &=& \# \{\mathrm{intersections~in~} \tau_i \cap \sigma_1 ~~\mathrm{~if~}\sigma_1 \mathrm{~travels ~
counterclockwise~ along~ } h_m\}\\
&+& \# \{\mathrm{intersections~in~} \tau_i \cap \sigma_2 ~~\mathrm{~if~}\sigma_2 \mathrm{~travels~
clockwise~ along~ }h_{m'}\} \\
&-& \# \{\mathrm{intersections~in~} \tau_i \cap \sigma_1 ~~\mathrm{~if~}\sigma_1 \mathrm{~travels~
clockwise~ along~ }h_m\} \\
&-& \# \{\mathrm{intersections~in~} \tau_i \cap \sigma_2 ~~\mathrm{~if~}\sigma_2 \mathrm{~travels~
counterclockwise~ along~ }h_{m'}\}.
\end{eqnarray*} 
We also make the convention that if $\alpha$ is a closed loop, then 
$\tilde{\ell}(\tilde{\rho}_{\alpha}, \tau_i)=0$.
\end{definition}

\begin{figure}
\input{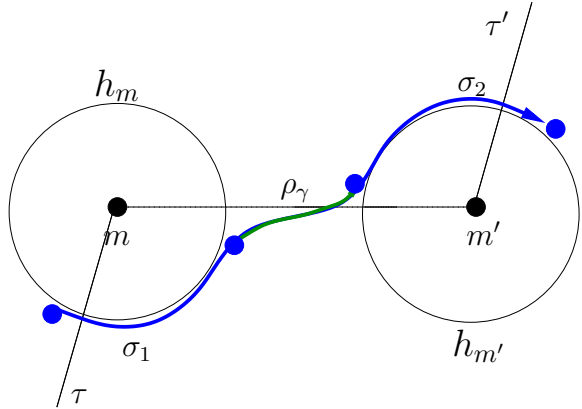}
\caption{A \emph{loosened $M$-path} with a \emph{positive} signed excess with respect to both $\tau$ and $\tau'$}\label{fig:sgn-excess}
\end{figure}

\begin{lemma}\label{skein-subs}
We use the notations of Definitions \ref{loosened} and \ref{signed-exc}.
Then we have that 
\begin{equation*} 
|\UR\left(\overline{M}(\tilde{\rho}_{\gamma})\right)|  = 
\overline{\chi}_{\gamma,T} 
\prod_{i=1}^n \Y_i^{\frac{-\tilde{\ell}(\tilde{\rho}_{\gamma}, \tau_i)}{2}}.
\end{equation*}
\end{lemma}
\begin{proof}
By Definition \ref{loosened}, 
the matrix 
$\overline{M}(\tilde{\rho}_{\gamma})$ associated to
a loosened $M$-path $\tilde{\rho}_{\gamma} = 
\sigma_2\circ \rho_\gamma \circ \sigma_1$
can be decomposed as a product 
$\overline{M}(\tilde{\rho}_{\gamma}) = \overline{M}(\sigma_2) \overline{M}(\rho_{\gamma}) \overline{M}(\sigma_1)$, where 
$\sigma_1$ and $\sigma_2$ consist only of steps of types 1 and 2.  
Therefore $\overline{M}(\sigma_1)$ and $\overline{M}(\sigma_2)$ are lower-triangular matrices in 
$SL_2(\R)$, so we obtain
\begin{eqnarray*} \left|\UR\left(\overline{M}(\tilde{\rho}_{\gamma})\right)\right| &=& \left|\UR\left(
\left[ \begin{matrix} \prod_{i=1}^n \Y_i^{b_i/2} & 0 \\ * & \prod_{i=1}^n \Y_i^{-b_i/2} \end{matrix}\right]
\overline{M}(\rho_\gamma) 
\left[ \begin{matrix} \prod_{i=1}^n \Y_i^{a_i/2} & 0 \\ * & \prod_{i=1}^n \Y_i^{-a_i/2} \end{matrix}\right]
\right)\right| \\  &=&  |\UR\left(\overline{M}(\rho_{\gamma})\right)| \prod_{i=1}^n \Y_i^{b_i/2-a_i/2} = \overline{\chi}_{\gamma,T} 
\prod_{i=1}^n \Y_i^{\frac{-\tilde{\ell}(\tilde{\rho}_{\gamma}, \tau_i)}{2}}.
\end{eqnarray*}
Here, the last equality follows from Definition \ref{signed-exc}. 
\end{proof}

For convenience, we also record the
following identities involving
$2 \times 2$ matrices.
If $m$ is a $2 \times 2$ matrix, let 
$\UR(m)$ denote its upper right entry, and $\tr(m)$ denote its trace.
These identities will be 
used to prove skein relations involving reduced matrices associated to 
generalized arcs and loops (as in Definition \ref{def:reduced_matrix}).

\begin{lemma}\label{matrix-identity}
Let $m_1, m_2, m_3$ be $2 \times 2$ matrices where $\det(m_1)=1$.
Then 
\begin{align}
\label{first}
\UR(m_2 m_1) \UR(m_1 m_3) &= \UR(m_1) \UR(m_2 m_1 m_3) + 
            \UR(m_2) \UR(m_3), \\
\label{second} \UR(m_3 m_2) \tr(m_1) &= \UR(m_3 m_1 m_2) + 
\UR(m_3 m_1^{-1} m_2), \text{ and }\\
\label{third} \tr(m_2) \tr(m_1) &= \tr (m_1 m_2) + \tr(m_1^{-1} m_2).
\end{align}
\end{lemma}
\begin{proof}
These identities can be easily checked by hand or  computer.
\end{proof}

In what follows, the notation $\sim$ denotes isotopy of curves.

\begin{prop} \label{prop:skein1}
Let $\gamma_1$ and $\gamma_2$ be two generalized arcs
which intersect each other at least once; let $x$ be a point of intersection;
and let $\alpha_1$, $\alpha_2$
and $\beta_1$, $\beta_2$ be the two pairs of arcs obtained by 
smoothing $\gamma_1$
and $\gamma_2$ at $x$, as in Figure \ref{fig:orient}.
In particular, we choose orientations and label these four arcs so that:
$$ \gamma_1  \sim 
\beta_1 \circ \alpha_2, ~~\gamma_2 \sim \alpha_1 \circ \beta_1, \text{ and } \beta_2 \sim \alpha_1 \circ \beta_1 \circ 
\alpha_2.$$

We then construct standard $M$-paths (Definition \ref{def:standard}) 
$\rho_{\alpha_1}$ from $s_1$ to $t_1$ and $\rho_{\alpha_2}$ from 
$s_2$ to $t_2$.  We let $\tilde{\rho}_{\beta_1}$ be a loosened $M$-path from $t_2$ to $s_1$ that extends a standard $M$-path corresponding to $\beta_1$.  Finally,
we choose loosened $M$-paths 
$\tilde{\rho}_{\gamma_1},~~
\tilde{\rho}_{\gamma_2},$ and 
$\tilde{\rho}_{\beta_2}$ so that 
\begin{equation} \label{skein1-isotopy}
\tilde{\rho}_{\gamma_1} \sim \tilde{\rho}_{\beta_1}\circ \rho_{\alpha_2},~~
\tilde{\rho}_{\gamma_2} \sim \rho_{\alpha_1}\circ \tilde{\rho}_{\beta_1},~~\mathrm{~ and~}~~ 
\tilde{\rho}_{\beta_2} \sim \rho_{\alpha_1}\circ \tilde{\rho}_{\beta_1}\circ \rho_{\alpha_2}.
\end{equation}
We then obtain the following identity:
\begin{equation} \label{skein-eq1}
\overline{\chi}_{\gamma_1,T}~ \overline{\chi}_{\gamma_2,T} = \pm \overline{\chi}_{\alpha_1,T} ~\overline{\chi}_{\alpha_2,T}
\prod_{i=1}^n \Y_i^{\frac{\tilde{\ell}(\tilde{\rho}_{\gamma_1}, \tau_i)
+ \tilde{\ell}(\tilde{\rho}_{\gamma_2}, \tau_i)}{2}}
\pm \overline{\chi}_{\beta_1,T} ~\overline{\chi}_{\beta_2,T} \prod_{i=1}^n \Y_i^{\frac{\tilde{\ell}(\tilde{\rho}_{\gamma_1}, \tau_i)
+ \tilde{\ell}(\tilde{\rho}_{\gamma_2}, \tau_i) - \tilde{\ell}(\tilde{\rho}_{\beta_1}, \tau_i)
- \tilde{\ell}(\tilde{\rho}_{\beta_2}, \tau_i)}{2}}.
\end{equation}

\end{prop}

\begin{prop}
\label{prop:skein2}
Let $\gamma_1$ be a generalized arc
or loop and let $\gamma_2$ be a generalized loop, such that
$\gamma_1$ and $\gamma_2$ intersect each other; let $x$ be a point of intersection.
Let $\alpha$ and $\beta$ be the two oriented generalized arcs or loops
obtained by 
{\it smoothing} $\gamma_1$
and $\gamma_2$ at $x$, as in Figure \ref{skein3figure}.  
We construct standard $M$-paths $\rho_{\gamma_1}$ and $\rho_{\gamma_2}$, respectively, so that 
they intersect at some $v_{m,\tau}^\pm$.  We let $\rho_{\gamma_1}^{(1)}$ be the first portion of the $M$-path $\rho_{\gamma_1}$ (up until $v_{m,\tau}^{\pm}$), and $\rho_{\gamma_1}^{(2)}$ be the second portion;
if $\gamma_1$ is a loop, then we let $\rho_{\gamma_1}^{(2)}$ denote the empty path.  We then choose
loosened $M$-paths (or just $M$-paths when $\gamma_1$ is a closed loop) 
$\tilde{\rho}_{\alpha}$ and $\tilde{\rho}_{\beta}$
for $\alpha$ and $\beta$ such that
$\tilde{\rho}_{\alpha} \sim \rho_{\gamma_1}^{(2)} \circ \rho_{\gamma_2} \circ \rho_{\gamma_1}^{(1)}$, 
and 
$\tilde{\rho}_{\beta} \sim \rho_{\gamma_1}^{(2)} \circ (\rho_{\gamma_2})^{-1} \circ \rho_{\gamma_1}^{(1)}$.
We have the identity
\begin{equation} \label{skein-eq2}
\overline{\chi}_{\gamma_1,T} ~\overline{\chi}_{\gamma_2,T} =  \pm \overline{\chi}_{\alpha,T} 
\prod_{i=1}^n \Y_i^{\frac{ - \tilde{\ell}(\tilde{\rho}_{\alpha}, \tau_i)}{2}}
\pm \overline{\chi}_{\beta,T}\prod_{i=1}^n \Y_i^{ \frac{ -\tilde{\ell}(\tilde{\rho}_{\beta}, \tau_i)}{2}}.
\end{equation}
\end{prop}

\begin{prop}
\label{prop:skein3}
Let $\gamma$ be a generalized arc or closed loop with a self-intersection at $x$.
Let $\alpha_1$, $\alpha_2$, and $\beta$ be the generalized arcs and loops
obtained by 
{\it smoothing} at $x$, as in Figure \ref{skein5figure}.
Construct the standard $M$-path $\rho_\gamma$ for $\gamma$, and write it as 
$\rho_\gamma^{(3)} \circ \rho_\gamma^{(2)} \circ \rho_{\gamma}^{(1)}$, where the 
subpaths meet at point $x$.  (If $\gamma$ is a loop, we let 
$\rho_{\gamma}^{(3)}$ be the empty path.)
We define an $M$-path $\tilde{\rho}_{\alpha_1} = \rho_{\gamma}^{(2)}$,
and choose loosened $M$-paths $\tilde{\rho}_{\alpha_2}$ and $\tilde{\rho}_\beta$ which are isotopic to 
$\rho_{\gamma}^{(3)} \circ \rho_{\gamma}^{(1)}$ and $\rho_{\gamma}^{(3)} \circ (\rho_{\gamma}^{(2)})^{-1} \circ \rho_{\gamma}^{(1)}$, respectively.
Then we have the identity
\begin{equation}\label{skein-eq3}
\overline{\chi}_{\gamma,T}  = 
\pm \overline{\chi}_{\alpha_1,T} \overline{\chi}_{\alpha_2,T} 
\prod_{i=1}^n \Y_i^{\frac{ - \tilde{\ell}(\tilde{\rho}_{\alpha_2}, \tau_i)}{2}}
\pm \overline{\chi}_{\beta,T}\prod_{i=1}^n \Y_i^{\frac{ - \tilde{\ell}(\tilde{\rho}_{\beta}, \tau_i)}{2}}.
\end{equation}
\end{prop}

\begin{remark}
In the above three propositions, as well as Propositions \ref{up:skein1}, \ref{up:skein2}, \ref{up:skein3}, and Corollary \ref{cor:skein}, the signs should be positive; see \cite[Section 12]{FG1} and \cite{Th-Note}.
\end{remark}

\begin{proof}
We start by proving Proposition \ref{prop:skein1}.  
Let $m_1 = \overline{M}(\tilde{\rho}_{\beta_1})$, $m_2 = \overline{M}(\rho_{\alpha_1})$, and $m_3 = \overline{M}(\rho_{\alpha_2})$.  By \eqref{skein1-isotopy} and 
Remark \ref{ReducedMats}, we have that 
\begin{equation}\label{matrix-partial}
m_1 m_3 = \pm\overline{M}(\tilde{\rho}_{\gamma_1}),~~~
m_2 m_1 = \pm\overline{M}(\tilde{\rho}_{\gamma_2}), \mathrm{~and~}~~ 
m_2 m_1 m_3 = \pm\overline{M}(\tilde{\rho}_{\beta_2}).
\end{equation}
By \eqref{first},
it follows that 
\begin{equation*}\label{ur-equation}
{\small |\UR\left(\overline{M}(\tilde{\rho}_{\gamma_1})\right)\UR\left(\overline{M}(\tilde{\rho}_{\gamma_2})\right)|
= \pm |\UR\left(\overline{M}(\rho_{\alpha_1})\right)\UR\left(\overline{M}(\rho_{\alpha_2})\right)|
~~\pm~~ |\UR\left(\overline{M}(\tilde{\rho}_{\beta_1})\right)\UR\left(\overline{M}(\tilde{\rho}_{\beta_2})\right)|.}
\end{equation*}
Applying Lemma \ref{skein-subs} 
to each of $\tilde{\rho}_{\gamma_1},~
\tilde{\rho}_{\gamma_2}$, and 
$\tilde{\rho}_{\beta_2}$, and substituting into the above equation, 
we obtain  
\eqref{skein-eq1}, as desired.

We now  prove Proposition \ref{prop:skein2}.  
Let $m_1 = \overline{M}(\rho_{\gamma_2})$, $m_2 = \overline{M}(\rho_{\gamma_1}^{(1)})$, and $m_3 = \overline{M}(\rho_{\gamma_1}^{(2)})$.
In the case that $\gamma_1$ is a loop, $\alpha$ and $\beta$ are also loops,
and $\tilde{\ell}(\rho_{\alpha},\tau_i) = 
\tilde{\ell}(\rho_{\beta},\tau_i) = 0$
for any arc $\tau_i \in T$.
Since in this case, $\tilde{\rho}_{\alpha}$ and $\tilde{\rho}_{\beta}$ are 
$M$-paths (not loosened $M$-paths), we obtain  \eqref{skein-eq2} immediately from \eqref{third}.
In the case that  $\gamma_1$ is a generalized arc, we have that
$\tilde{\rho}_{\alpha}$ and $\tilde{\rho}_{\beta}$ are loosened $M$-paths, so 
we apply Lemma \ref{skein-subs}  to
$|\UR\left(\overline{M}(\tilde{\rho}_{\alpha})\right))|$ and $|\UR\left(\overline{M}(\tilde{\rho}_{\beta})\right))|$.  
We then obtain \eqref{skein-eq2} from \eqref{second}.

The proof of Proposition \ref{prop:skein3} is analogous.  
We set $m_1 = \overline{M}(\rho_{\gamma}^{(2)}),$
$m_2 = \overline{M}(\rho_{\gamma}^{(1)}),$ and 
$m_3 = \overline{M}(\rho_{\gamma}^{(3)})$.  
If $\gamma$ is a generalized arc, then from \eqref{second} we have that 
$\UR(m_3 m_1 m_2)  = 
\UR(m_3 m_2) \tr(m_1) - 
\UR(m_3 m_1^{-1} m_2).$
And if $\gamma$ is a closed loop, then from \eqref{third} we have that
$\tr (m_1 m_2) = 
\tr(m_2) \tr(m_1) - 
\tr(m_1^{-1} m_2).$ 
Applying  Lemma \ref{skein-subs} gives
\eqref{skein-eq3}.
\end{proof}

We can use  the previous three propositions to obtain formulas for the 
quantities $\hat{\chi}_{\gamma,T}$. 
Before stating these formulas, we introduce a variant of intersection numbers that we refer to as the \emph{signed intersection number} between a (loosened) $M$-path $\rho$ and an arc $\tau \in T$. 

Recall from Definition \ref{e_p} that $e(\gamma,\tau)$ denotes the \emph{intersection number} between $\gamma$ and $\tau$, i.e. the number of crossings between the generalized arc $\gamma$ and the arc $\tau$.

\begin{definition}
Given a (loosened) $M$-path $\tilde{\rho}_\gamma$ for a generalized arc or loop $\gamma$,
and an arc $\tau$, we define the \emph{signed intersection number} between $\tilde{\rho}_\gamma$ and $\tau$ to be $$\ell(\tilde{\rho}_\gamma,\tau) = e(\gamma,\tau) + \tilde{\ell}(\tilde{\rho}_\gamma,\tau).$$ 
\end{definition}

\begin{remark} \label{FT-dict}
Our definitions of signed intersection numbers and signed excess for loosened $M$-paths are motivated by the definition of transverse measures appearing in \cite[Sec. 13]{FT}.  In particular,
when $T$ has no self-folded triangles, the signed intersection number $\ell(\tilde{\rho}_\gamma,\tau)$ between a loosened $M$-path $\tilde{\rho}_\gamma$ corresponding to an arc $\gamma$ and an arc $\tau \in T$ is equal to the \emph{transverse measure} $\ell_{\overline{L}_\tau}(\overline{\gamma})$ between a \emph{lift} $\overline{\gamma}$ of the arc $\gamma$ to an \emph{opened surface} and a lift $\overline{L}_\tau$ of the elementary lamination corresponding to arc $\tau \in T$.  There is a direct connection between these two quantities if $T$ has self-folded triangles as well, but the formula relating the two is more complicated.  In particular, in this case we have a multi-lamination containing two elementary laminations spiralling into a puncture (one clockwise, one counterclockwise), but the transverse measures associated to such laminations do not match up as simply to the signed intersection numbers which are defined in this paper.
\end{remark}

\begin{corollary}\label{cor:skein}
Using the notation  of Proposition \ref{prop:skein1}, we have 
{\small \begin{align*} \label{skein-eq4}
{\chi}_{\gamma_1,T}~ {\chi}_{\gamma_2,T} = &\pm {\chi}_{\alpha_1,T} ~{\chi}_{\alpha_2,T}
\prod_{i=1}^n \Y_i^{\frac{\ell(\tilde{\rho}_{\gamma_1}, \tau_i)
+ \ell(\tilde{\rho}_{\gamma_2}, \tau_i) - \ell(\rho_{\alpha_1}, \tau_i)
- \ell(\rho_{\alpha_2}, \tau_i)}{2}}\\
& \pm {\chi}_{\beta_1,T} ~{\chi}_{\beta_2,T} \prod_{i=1}^n \Y_i^{\frac{\ell(\tilde{\rho}_{\gamma_1}, \tau_i)
+ \ell(\tilde{\rho}_{\gamma_2}, \tau_i) - \ell(\tilde{\rho}_{\beta_1}, \tau_i)
- \ell(\tilde{\rho}_{\beta_2}, \tau_i)}{2}}. 
\end{align*}}

Using the notation  of Proposition \ref{prop:skein2}, we have 
\begin{equation*} \label{skein-eq5}
{\chi}_{\gamma_1,T} ~{\chi}_{\gamma_2,T} =  \pm {\chi}_{\alpha,T} 
\prod_{i=1}^n \Y_i^{\frac{\ell(\rho_{\gamma_1},\tau_i) + \ell(\rho_{\gamma_2},\tau_i) - \ell(\tilde{\rho}_{\alpha}, \tau_i)}{2}}
\pm {\chi}_{\beta,T}\prod_{i=1}^n \Y_i^{ \frac{\ell(\rho_{\gamma_1},\tau_i) + \ell(\rho_{\gamma_2},\tau_i)
 -\ell(\tilde{\rho}_{\beta}, \tau_i)}{2}}.
\end{equation*}

Using the notation of of Proposition \ref{prop:skein3}, we have 
\begin{equation*}\label{skein-eq6}
{\chi}_{\gamma,T}  = 
\pm {\chi}_{\alpha_1,T} {\chi}_{\alpha_2,T} 
\prod_{i=1}^n \Y_i^{\frac{\ell(\rho_{\gamma},\tau_i) - \ell(\tilde{\rho}_{\alpha_1}, \tau_i)
- \ell(\tilde{\rho}_{\alpha_2}, \tau_i)}{2}}
\pm {\chi}_{\beta,T}\prod_{i=1}^n \Y_i^{\frac{\ell(\rho_{\gamma},\tau_i) - \ell(\tilde{\rho}_{\beta}, \tau_i)}{2}}.
\end{equation*}
\end{corollary}

\begin{proof}
Since we have assumed that $T$ has no self-folded triangles,
we have that $\chi_{\gamma,T} = \hat{\chi}_{\gamma,T}$.
Now it follows from Definition \ref{def:chi} that when we substitute the quantities $\hat{\chi}_{\gamma,T}$ for $\overline{\chi}_{\gamma,T}$ 
in Propositions \ref{prop:skein1}, \ref{prop:skein2} and \ref{prop:skein3},  
we get an extra factor of  $\sqrt{y_{\tau_i}}$ for every elementary step of type 2 in the 
corresponding $M$-path.  This explains the appearance of the $e(\gamma,\tau_i)/2$'s (implicitly via the signed intersection numbers) in the exponents of the new identities.  Corollary
\ref{cor:skein} then follows from \eqref{skein-eq1}, \eqref{skein-eq2}, and \eqref{skein-eq3}.
\end{proof}

\begin{remark}
Note that the coefficients in Corollary \ref{cor:skein} seem to depend on our choices of $M$-paths and 
loosened $M$-paths.  
However, it follows from \cite[(13.10)]{FT} (which cites \cite{FG1, FG3} and 
well-known facts about intersection numbers)
that in fact the coefficients do not depend on our choices.  
To apply \cite[(13.10)]{FT} to our situation, we use the dictionary indicated in Remark \ref{FT-dict} in the case when $T$ has no self-folded triangles.  Though the signed intersection numbers or signed excesses \emph{do} depend on the choices of loosened $M$-paths, the differences between the signed intersection numbers appearing in the exponents of these expressions are independent of these choices.  
\end{remark}

\begin{lemma} The coefficients appearing in Corollary \ref{cor:skein} 
are Laurent monomials in the variables $y_1,\dots,y_n$.
\end{lemma}

\begin{proof}
It suffices to show that the exponent of each $y_i$ in Corollary \ref{cor:skein}
is an integer (not a half-integer).
We first claim that for any loosened $M$-path $\tilde{\rho}_{\gamma}$,
$\ell(\tilde{\rho}_{\gamma}, \tau) \equiv
 e(\tilde{\rho}_{\gamma},\tau) \mod 2$.
To see this, write 
$\tilde{\rho}_{\gamma} = \sigma_2 \circ \rho_{\gamma} \circ \sigma_1$,
and let $a = e(\sigma_1, \tau)$,
$b=e(\sigma_2,\tau)$, and $x=e(\rho_{\gamma}, \tau)$.
Then $e(\tilde{\rho}_{\gamma},\tau) = x+a+b$
and $\ell(\tilde{\rho}_{\gamma},\tau) = x\pm a\pm b$.  It's clear
that both $x+a+b$ and $x \pm a \pm b$ have the same parity.

Now observe that if $C$ (as in Definition \ref{def:smoothing})
is a pair of curves with an intersection or 
one curve with a self-intersection, and we smooth $C$ at a
point $x$, obtaining $C_+$, then 
$e(C,\tau) - e(C_+,\tau)$ is a (positive) even integer.
Here, if $C = \{\gamma_1,\gamma_2\}$, then 
$e(C,\tau)$ denotes $e(\gamma_1,\tau)+e(\gamma_2,\tau)$.
The lemma follows.
\end{proof}

We now turn back to the unpunctured surface case, and Propositions \ref{up:skein1}, \ref{up:skein2}, and \ref{up:skein3}.  In this situation, every marked point $m \in M$ is on the boundary of $S$, and thus we have an arc segment $h_m \cap S$ (rather than a circle $h_m$) associated to each $m \in M$.  We define some new notation that will be useful in the proofs of these propositions.

\begin{figure}
\input{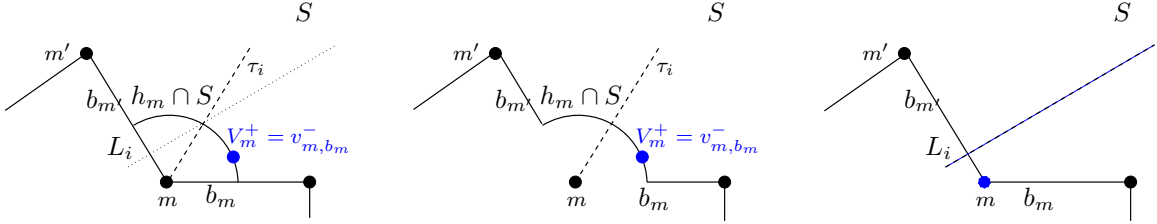}
\caption{(Left): Illustrating a horocyclic arc segment at a boundary.  (Middle and Right): Comparing the local configuration of $\tau_i$ and $V_m^+$ with the local configuration 
of $L_i$ and the marked point $m$. }
\label{fig HoroBound}
\end{figure}

Given a 
boundary component $C_m$ containing $m$, we let $b_m$ denote the boundary segment immediately before and incident to $m$ when traveling along $C_m$ in the counterclockwise direction.  We let $V_m^+$ denote the clockwise-most vertex on $h_m \cap S$, which also has the form $v_{m,b_m}^-$.  See Figure \ref{fig HoroBound}.  
With this notation in mind, we now prove Propositions 
\ref{up:skein1},  
\ref{up:skein2}, and   
\ref{up:skein3} together, using   
Corollary \ref{cor:skein}.

\begin{proof} 
The main idea of this proof is that for each arc
$\alpha$ involved in one of the skein relations, 
we can choose a loosened $M$-path 
$\tilde{\rho}_{\alpha}$ such that 
$\ell(\tilde{\rho}_{\alpha}, \tau_i) 
= e(\alpha, L_i)$.  
In particular, we choose loosened $M$-paths $\tilde{\rho}$ so that they start and end at a point in the set $\{V_m^+\}_{m \in M}$.  For a generalized arc $\gamma$, this completely determines the starting and ending points of the associated loosened $M$-path $\tilde{\rho}_{\gamma}$.  
As above, we decompose such a loosened $M$-path $\tilde{\rho}_{\gamma}$ as $\sigma_2 \circ \rho_{\gamma} \circ \sigma_1$ where $\rho_{\gamma}$ is an $M$-path, and $\sigma_i$ consists only of steps of type $1$ and $2$.

First, observe that in this situation 
the signed excess $\tilde{\ell}(\tilde{\rho}_{\gamma},\tau_i)$ is nonnegative and equal to $e(\sigma_1,\tau_i) + e(\sigma_2,\tau_i)$ for any generalized arc $\gamma$ and any arc $\tau_i \in T$;  
see Figure \ref{fig:sgn-excess}.
Second, 
observe from Figure \ref{fig HoroBound} 
that the relationship between $V_m^+$ and $\tau_i$  is analogous to the 
relationship between $m$ and $L_i$:
in particular, any arc $\tau_i \in T$ incident to $m$ lies in the counterclockwise direction (along $C_m$) from $V_m^+$ just as the corresponding elementary lamination $L_i$ lies in the counterclockwise
direction from $m$.  Consequently, for any generalized arc or loop $\gamma$ and arc $\tau_i \in T$, we have the identity
\begin{equation}
\label{up-identity}
\ell(\tilde{\rho}_{\gamma},\tau_i) = \tilde{\ell}(\tilde{\rho}_{\gamma},\tau_i) + e(\gamma,\tau_i) = 
e(\sigma_1,\tau_i) + e(\sigma_2,\tau_i) + e(\gamma,\tau_i)
= e(\gamma, L_i).
\end{equation}
Plugging \eqref{up-identity} into Corollary \ref{cor:skein} completes the proof of all three propositions.
\end{proof}

See Figure \ref{fig Mbar} for an example; note that of the three laminations listed, $L_{\gamma_1}$ is the only lamination that crosses $\{\gamma_1,\gamma_2\}$ and $\{\beta_1,\beta_2\}$ but not  
$\{\alpha_1,\alpha_2\}$.  Analogously the concatenation $\tilde{\rho}_{\alpha_2} \circ \tilde{\rho}_{\beta_1}$ can be pulled taut as to avoid intersections with $\beta_1$ and $\beta_2$.  This yields the loosened $M$-path $\tilde{\rho}_{\gamma_1}$.  However, the loosened $M$-paths $\tilde{\rho}_{\alpha_1}$ and $\tilde{\rho}_{\beta_1}$ must intersect both the pairs $\{\alpha_1,\alpha_2\}$ and $\{\beta_1,\beta_2\}$.  This is consistent with the exchange relation
$$x_{\gamma_1} x_{\gamma_1'} = \Y_{\gamma_1}x_{\alpha_1}x_{\alpha_2} + x_{\beta_1}x_{\beta_2}$$ that we obtain by shear coordinates in this example.

\begin{figure}
\input{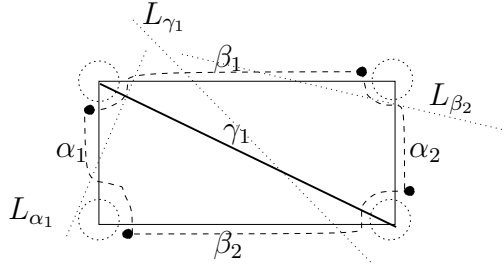}
\caption{Comparing shear coordinates with respect to elementary laminations to intersections with loosened $M$-paths}\label{fig Mbar}
\end{figure}

{} 

\end{document}